\documentclass [11pt] {article}
\usepackage{amsfonts}
\usepackage{amssymb}
\usepackage{times}
\usepackage{graphicx}

\newcommand{\qed} {\hspace {0.1in} \rule {1.5mm} {3.5mm}}

\newtheorem{lemma}{Lemma}[section]
\newtheorem{corollary}{Corollary}[section]

\newtheorem{theorem}{Theorem}

\newtheorem{proposition}{Proposition}[section]
\newtheorem{definition}{Definition}[section]
\def\beQ{{\bf Q}}

\def\o{\omega}

\def\limn{\lim_{n\to\infty}}

\def\e{\epsilon}

\def\limo{\lim_{\omega}}

\def\<{\langle}
\def\>{\rangle}

\def\proof{\smallskip\noindent{\it Proof.} }

\def\bE{{\bf E}}
\def\bH{{\bf H}}
\def\bI{{\bf I}}
\def\bJ{{\bf J}}
\def\bR{{\mathbb R}}
\def\bN{{\mathbb N}}\def\bQ{{\mathbb Q}}

\def\cA{\mbox{$\cal A$}}
\def\cB{\mbox{$\cal B$}}
\def\cC{\mbox{$\cal C$}}

\def\cH{\mbox{$\cal H$}}

\def\cF{\mbox{$\cal F$}}

\def\cA{\mbox{$\cal A$}}
\def\cB{\mbox{$\cal B$}}
\def\cC{\mbox{$\cal C$}}

\def\cP{\mbox{$\cal P$}}

\def\cM{\mbox{$\cal M$}}
\def\cN{\mbox{$\cal N$}}

\def\to{\rightarrow}

\def\xo{{\bf X}}
\def\bo{\cB_\omega}

\def\muo{\mu}

\def\xok{\xo^{[k]}}
\def\ac{A^c}
\def\wch{\widetilde{\cH}}
\def\wp{\widetilde{P}}
\def\bbx{{\bf x}}
\def\fb{{\bf f}}
\def\gb{{\bf g}}

\begin{document}

\title{A measure-theoretic
 approach to the theory of dense hypergraphs}
\author{{\sc G\'abor Elek} and {\sc Bal\'azs Szegedy}}

\maketitle

\abstract{ In this paper we develop a measure-theoretic method to treat
 problems in hypergraph theory. Our central theorem is a correspondence 
principle between three objects: An increasing hypergraph sequence, 
a measurable set in an ultraproduct space and a measurable set in a 
finite dimensional Lebesgue space.
Using this correspondence principle we build up the theory of
 dense hypergraphs from scratch. Along these lines we give 
new  proofs for the Hypergraph Removal Lemma, the Hypergraph 
Regularity Lemma, the Counting Lemma and the Testability of 
Hereditary Hypergraph Properties.
We prove various new results including a strengthening of the 
Regularity Lemma and an Inverse Counting Lemma. We also prove the
 equivalence of various notions for convergence of hypergraphs and 
we construct limit objects for such sequences. We prove that the 
limit objects are unique up to a certain family of measure preserving 
transformations.
As our main tool we study the integral and measure theory on the 
ultraproduct of finite measure spaces which is interesting on its own right.
}

\newpage

\tableofcontents

\section{Introduction}
The so-called Hypergraph Regularity Lemma (R\"odl-Skokan \cite{RSko},
R\"odl-Schacht \cite{RS},
Gowers \cite{Gow}, later generalized by Tao \cite{Tao}) is one of the
most exciting result
in modern combinatorics. It exists in many different forms, strength
and generality. The main message in all of them is that every
$k$-uniform hypergraph can be approximated by a structure which
consists of boundedly many random-looking (quasi-random) parts for any
given error $\epsilon$. Another common feature of these theorems is
that they all come with a corresponding counting lemma \cite{NRS}
which describes how to estimate the frequency of a given small
hypergraph from the quasi-random approximation of a large hypergraph.
One of the most important applications of this method is that it
implies the Hypergraph Removal Lemma (first proved by
Nagle, R\"odl and Schacht \cite{NRS})  and by an
observation of Solymosi \cite{S} it also
implies Szemer\'edi's celebrated theorem on
arithmetic progressions in dense subsets of the integers even in a
multidimensional setting.

 In this paper we present an analytic approach to the subject.
First, for any given sequence of hypergraphs we associate the
so-called ultralimit hypergraph,
which is a measurable hypergraph in a large (non-separable)
probability measure space.
The ultralimit method enables us to convert theorems of finite
combinatorics to measure theoretic statements on our ultralimit
space. In the second step, using separable factors
we translate these measure-theoretic
theorems to well-known results on the more familiar Lebesgue
spaces.
 
The paper is built up in a way that these two steps are compressed into
a {\bf correspondence principle} between the following three objects
\begin{enumerate}
\item An increasing sequence $\{H_i\}_{i=1}^\infty$ of $k$-uniform hypergraphs
\item The ultraproduct hypergraph $\bH\subseteq\xo^k$.
where $\xo$ is the ultraproduct of the vertex sets.
\item A measurable subset $W\subseteq [0,1]^{2^k-1}$.
\end{enumerate}
Using this single correspondence principle we are able to prove 
several results in hypergraph theory.
The next list is a summary of some of these results.

\begin{enumerate}
\item {\bf Removal lemma:} We prove the hypergraph removal lemma 
directly from Lebesgue's density theorem applied for the set 
$W\subseteq [0,1]^{2^k-1}$. In a nutshell, we convert the 
original removal lemma into the removal of the non-density points from $W$ 
which is a $0$-measure set. (Theorem \ref{removallemma})
    
\item {\bf Regularity lemma} We deduce the hypergraph regularity lemma from a
  certain finite box approximation of $W$ in $L_1$. To be more precise, 
$W$ is approximated by a set which is the disjoint union of finitely 
many direct product sets in $[0,1]^{2^k-1}$. ( Theorem \ref{regularitylemma})

\item {\bf Limit object} We prove that $W$ serves as a limit object for 
hypergraph sequences $\{H_i\}_{i=1}^\infty$ which are convergent in the sense
that the densities of every fixed hypergraph $F$ converge. 
Limits of $k$-uniform hypergraphs can also be represented by $2^k-2$
variable measurable functions $w:[0,1]^{2^k-2}\to[0,1]$ such that the
coordinates are indexed by the proper non-empty subsets of $\{1,2,\dots,k\}$
and $w$ is invariant under the induced action of $S_k$ on the coordinates. 
This generalizes a theorem by Lov\'asz and Szegedy. ( Theorem \ref{lobtetel})

\item {\bf Sampling and concentration:} Even tough $W$ is a measurable set, it
  makes sense to talk about random samples from $W$ which are ordinary
  hypergraphs.
 We prove concentration results for this sampling process. 
The sampling processes give rise to random hypergraph models which are 
interesting on their own right. ( Theorem \ref{conc} and Theorem \ref{asconv})

\item {\bf Testability of hereditary properties:} We give a new proof for the
  testability of hereditary hypergraph properties (This was 
first proved for graphs by Alon-Shapira and later for hypergraphs by
R\"odl-Schacht). The key idea is based on a modified sampling process 
from the limit object $W$ that we call ``hyperpartition sampling''. This
creates an overlay of samples from $W$ and the members of the 
sequence $\{H_i\}_{i=1}^\infty$ such that expected Hamming distance of $H_i$ 
and the corresponding sample is small. (Theorem \ref{testable})

\item {\bf Regularity as compactness:} We formulate a strengthening of the
  hypergraph regularity lemma which puts the regularity in the framework of
  compactness.
 Roughly speaking this theorem says that every increasing hypergraph sequence
 has a subsequence which converges in a very strong (structural) sense. 
 Here we introduce the notion of {\bf strong convergence}. (Theorem
 \ref{compa})

\item {\bf Distance notions:} We introduce several distance notions between
 hypergraph limit objects (and hypergraphs) and we analyze their relationship.
(Theorem \ref{uni2})

\item {\bf Uniqueness:} We prove the uniqueness of the limit object up to a
 family of measure preserving transformations on $[0,1]^{2^k-1}$. 
This generalizes a result of Borgs-Chayes-Lov\'asz from graphs to hypergraphs.
(Theorem \ref{uni1})
    
\item {\bf Counting Lemma:} We prove that the structure of 
regular partitions determine the \\ 
subhypergraph densities. (Theorem \ref{counting} and Corollary \ref{counting2})
    
\item {\bf Equivalence of convergence notions:} We prove 
that convergence and strong convergence are equivalent.
 For technical reasons we introduce a third convergence 
notion which is a slight variation of strong convergence and we call it {\bf
  structural convergence}. This is also equivalent with the other two
notions. 
The third notation enables us to speak about structural 
limit objects which turns out to be the same as the original limit object.
(Theorem \ref{three})
    
\item {\bf Inverse counting lemma:} Using the equivalence of 
convergence notions we obtain that if two hypergraphs have similar 
sub-hypergraph densities then they have similar regular partitions. 
In other words this means that regular partitions can be tested by 
sampling small hypergraphs. (Corollary \ref{inverse})
     
\end{enumerate}
{\bf Remark:} In our proofs we use the Axiom of Choice. However, G\"odel
in his seminal work {\it The Consistency of the Axiom of Choice
and the Generalized Continuum Hypothesis with the Axioms of Set Theory}
proved that (see also \cite{Chang}): If $\Gamma$ is an 
arithmetical statement and
$\Gamma$ is provable in {\bf ZF} with the
Axiom of Choice then $\Gamma$ is provable in {\bf ZF}.
 In fact, G\"odel
gave an algorithm to convert a formal {\bf ZFC}-proof of
 an arithmetical statement to a
{\bf ZF}-proof. An arithmetical statement is a statement in the form of
$$(\beQ_1 x_1 \beQ_2 x_2 \dots \beQ_k x_k) P(x_1,x_2,\dots,x_k)\,,$$
where the $\beQ_i$'s are existential or universal quantifiers and the
relation $P(x_1,x_2,\dots,x_k)$ can be checked by a Turing machine in finite
time.
The reader can convince himself that the Hypergraph Removal Lemma, The
Hypergraph Regularity Lemma, the Counting Lemma and the Inverse Counting Lemma
are all arithmetical statements.

\vskip0.2in
\noindent
{\bf Acknowledgement:} We are very indebted to
Terence Tao and L\'aszl\'o Lov\'asz for
helpful discussions.
\section{Preliminaries}\label{prelim}

\subsection{Homomorphisms, convergence and completion of hypergraphs}

\label{homconcomp}
Let $\mathcal{H}_k$ denote the set of isomorphism classes of
 finite $k$-uniform hypergraphs. For an element $H\in\mathcal{H}_k$
we denote the vertex set by $V(H)$ and the edge set by $E(H)$.
In this paper we view a $k$-uniform hypergraph $H$ on the vertex set $V$
as a subset of $V^k$ without having
repetitions in the coordinates and being invariant under the action of the
symmetric group $S_k$. Let $v_1,v_2,v_3,\dots,v_{|V|}$ be the elements
of $V$. Then an edge $E\in E(H)$ is a subset of $k$-elements
$\{v_{i_1}, v_{i_2},\dots,
v_{i_k}\}\subset V$ such that $(v_{i_1}, v_{i_2},\dots,
v_{i_k})\in H$. If $L$ is a family of edges in $H$, then $\hat{L}$ denote the
set of elements $(v_{i_1}, v_{i_2},\dots,
v_{i_k})\in H$ such that $\{v_{i_1}, v_{i_2},\dots,
v_{i_k}\}\in L$.

\begin{definition}
A {\bf homomorphism} between two elements $F,H\in\mathcal{H}_k$ is a
map $f:V(F)\mapsto V(H)$ which maps edges of $F$ into edges of $H$.
We denote by $\hom(F,H)$ the number of homomorphisms from $F$ to $H$
and by $\hom^0(F,H)$ the number of injective homomorphisms.
An {\bf induced homomorphism} is a map $f:V(G)\mapsto V(H)$ which maps edges to edges and non edges to non edges.
\end{definition}

Note that in the definition of $\hom$ the map $V(F)\mapsto V(H)$ does not
have to be injective but the definition implies that it is injective if we
restrict it to any edge of $F$. There is a simple inclusion-exclusion type
formula which computes $\hom^0$ from $\hom$. To state this we will need
some more definitions.

Let $\mathcal{P}=\{P_1,P_2,\dots,P_s\}$ be a partition of $V(F)$ and let
 $f:V(F)\mapsto\mathcal{P}$ be the function which maps each vertex to its
partition set. We define a hypergraph $F(\mathcal{P})$ whose vertex set is
$\mathcal{P}$ and the edge set is $f(E(F))$. Note that $F(\mathcal{P})$ is a
$k$-uniform hypergraph if and only if every partition set intersect every
edge in at most $1$ element. We define the height $h(\mathcal{P})$ of
$\mathcal{P}$ as $|V(F)|-|\mathcal{P}|$.

\begin{lemma}\label{hom1} If $F$ and $H$ are $k$-uniform hypergraphs then
$$\hom(F,H)=\sum_{\mathcal{P}}\hom^0(F(\mathcal{P}),H)$$
and
$$\hom^0(F,H)=\sum_{\mathcal{P}}(-1)^{h(\mathcal{P})}\hom(F(\mathcal{P}),H)$$
where $\mathcal{P}$ runs through all partitions of $V(F)$
and $\hom(F(\mathcal{P}),H)$ and $\hom^0(F(\mathcal{P}),H)$ are defined to
be $0$ if $F(\mathcal{P})$ is not $k$-uniform.
\end{lemma}

\begin{proof}
The first equation is obvious from the definitions.
It implies that for any partition $\mathcal{P}$
we have that
$$\hom(F(\mathcal{P}),H)=
\sum_{\mathcal{P}'\leq\mathcal{P}}\hom^0(F(\mathcal{P}'),H)$$
where the sum runs through all partitions $\mathcal{P}'$ such that
$\mathcal{P}$ is a refinement of $\mathcal{P}'$. The inversion formula for
the partition lattice yields the second equation.
\end{proof}

Now we are ready to prove the next lemma.

\begin{lemma}\label{hom2} If $H_1,H_2\in\mathcal{H}_k$ are two hypergraphs
such that $\hom(F,H_1)=\hom(F,H_2)$ for every element $F\in\mathcal{H}_k$
then $H_1$ and $H_2$ are isomorphic.
\end{lemma}

\begin{proof} Lemma \ref{hom1} implies that $\hom^0(F,H_1)=\hom^0(F,H_2)$
for all hypergraphs $F\in\mathcal{H}_k$. In particular $\hom^0(H_1,H_2)=
\hom^0(H_1,H_1)>0$ and $\hom^0(H_2,H_1)=\hom^0(H_2,H_2)>0$ which implies
that $|V(H_1)|=|V(H_2)|$ and $|E(H_1)|=|E(H_2)|$. We obtain that every
injective homomorphism from $H_1$ to $H_2$ is an isomorphism. Since
 such a homomorphism exists the proof is complete.
\qed \end{proof}

The next two definitions will be crucial.

\begin{definition} The {\bf homomorphism density} $t(F,H)$ denotes the
probability that a random map $f:V(F)\mapsto V(H)$ is a homomorphism.
It can also be defined by the equation
$$t(F,H)=\frac{\hom(F,H)}{|V(H)|^{|V(F)|}}.$$
We also define $t_{\rm ind}(F,G)$ which is the probability that a random map
$f:V(F)\mapsto V(H)$ is an induced homomorphism. Finally $t^0_{\rm ind}(F,H)$ 
denotes the probability that a random injective map is an induced homomorphism.
\end{definition}

\begin{definition} A $t$-fold {\bf equitable blowup} of a hypergraph
$H\in H_k$ is a hypergraph $H'$ which is obtained by replacing each
vertex of $H$ by $t$ new vertices and each edge of $H$ by a complete
$k$-partite hypergraph on the corresponding new vertex sets.
\end{definition}

It is clear that if $H'$ is a $t$-fold equitable blowup of $H$ then
$\hom(F,H')=\hom(F,H)t^{|V(F)|}$ and consequently $t(F,H)=t(F,H')$.
The next lemma shows that hypergraphs from $\mathcal{H}_k$ are
 ``essentially'' separated by homomorphism densities except that
equitable blowups of a hypergraph cannot be separated.

\begin{lemma}\label{hom3} Let $H_1,H_2\in\mathcal{H}_k$ be two hypergraphs
  and assume that $t(F,H_1)=t(F,H_2)$ for every $F\in\mathcal{H}_k$. Then there
  exists a
$H\in\mathcal{H}_k$ which is an equitable blowup of both $H_1$ and $H_2$.
\end{lemma}

\begin{proof} Let $H_1'$ be the $|V(H_2)|$-fold equitable blowup of $H_1$ and
  let $H_2'$ be the $|V(H_1)|$-fold equitable blowup of $H_2$. Then
  $$|V(H_1')|=
|V(H_2')|=|V(H_1)||V(H_2)|$$ and
$t(F,H_1')=t(F,H_2')$ for every $F\in\mathcal{H}_k$. We obtain that
$$\hom(F,H_1')=t(F,H_1')|V(H_1')|^{|V(F)|}=t(F,H_2')|V(H_2)'|^{|V(F)|}=
\hom(F,H_2')$$
for every $F\in\mathcal{H}_k$. By Lemma \ref{hom2} the proof is complete.
\qed\end{proof}

The previous lemma motivates the following definition

\begin{definition} Two hypergraphs $H_1,H_2\in\mathcal{H}_k$ will be called
{\bf density equivalent} if there exists
 $H\in\mathcal{H}_k$ which is an equitable
blowup of both $H_1,H_2$ or equivalently, by Lemma \ref{hom3},
$t(F,H_1)=t(F,H_2)$ for every $F\in\mathcal{H}_k$.
\end{definition}

Homomorphism densities can be used to define two convergence
 notions on the set $\mathcal{H}_k$ which are slight variations of each other.

\begin{definition}A hypergraph sequence $\{H_i\}_{i=1}^\infty$ in
  $\mathcal{H}_k$ is called {\bf convergent} if $$\lim_{i\to\infty}t(F,H_i)$$
  exists
for every $F\in\mathcal{H}_k$. We say that $\{H_i\}_{i=1}^\infty$ is {\bf
  increasingly convergent} if it is convergent and
$$\lim_{i\to\infty}|V(H_i)|=
\infty.$$
\end{definition}

Both convergence notions lead to a completion of the set $\mathcal{H}_k$. We
denote the first completion by $\bar{\mathcal{H}}_k$ and the second one by
$\hat{\mathcal{H}}_k$. These two spaces are very closely related to each
other. It will turn out that $\bar{\mathcal{H}}_k$ is arc-connected whereas
$\hat{\mathcal{H}}_k$ is the union of $\mathcal{H}_k$ with the discrete
topology and $\bar{\mathcal{H}}_k$. In the space $\hat{\mathcal{H}}_k$ the set
$\bar{\mathcal{H}}_k$ behaves as a ``boundary'' for the set
$\mathcal{H}_k$. An advantage of the set $\hat{\mathcal{H}_k}$ is that it
directly contains the familiar set $\mathcal{H}_k$ of hypergraphs. A
disadvantage of $\hat{\mathcal{H}}_k$ is that it is not connected.
On the other hand $\bar{\mathcal{H}}_k$ is connected and $k$-uniform
hypergraphs are represented in it up to dense equivalence. In this paper we
focus only on $\bar{\mathcal{H}}_k$ so we give a
 precise definition only of this space.

Let $\delta$ be the following metric on $\mathcal{H}_k$.
 For two elements $H_1,H_2\in\mathcal{H}_k$ we define $\delta(H_1,H_2)$ as the
 infimum of the numbers $\epsilon\geq 0$ for
which $|t(F,H_1)-t(F,H_2)|\leq\epsilon$ holds for all $F\in\mathcal{H}_k$ with
$|V(F)|\leq 1/\epsilon$. Two hypergraphs have
$\delta$-distance zero if and only if they are density equivalent.
We denote the completion of this metric space by $\bar{\mathcal{H}_k}$.

The elements of the space $\bar{\mathcal{H}_k}$ have many interesting
representations. We give here one which is the most straightforward.
 Let $\mathcal{M}_k$ denote the compact space $[0,1]^{\mathcal{H}_k}$. Every
 graph $H\in\mathcal{H}_k$ can be represented as a point in
$\mathcal{M}_k$ by the sequence $T(H)=\{t(F,H)\}_{F\in\mathcal{H}_k}$. By
Lemma \ref{hom3} the point set $T(\mathcal{H}_k)$ represents
the density equivalence classes of $k$-uniform hypergraphs. The closure of
$T(\mathcal{H}_k)$ in $\mathcal{M}_k$ is a representation of
$\bar{\mathcal{H}_k}$. This representation shows immediately that
$\bar{\mathcal{H}_k}$ is compact since it is a closed subspace of the
compact space $\mathcal{M}_k$. To see that $\bar{\mathcal{H}_k}$ is
arc-connected requires some more effort, but it will follow easily
from one of our results in this paper (Theorem \ref{lobtetel}.)

An important feature of the space $\bar{\mathcal{H}}_k$ is that it makes sense
to talk about homomorphism densities of the form $t(F,X)$
if $X\in\bar{\mathcal{H}}_k$ and $F\in\mathcal{H}_k$.

We will denote by $[n]$ the set $\{1,2,\dots,n\}$.
For a subset $B\subset [k]$, $r(B)$ will stand for the non-empty subsets of
$B$. Similarly, $r([n],k)$ will denote the set of all non-empty subsets
of $[n]$ having size at most $k$.
If $K$ is a hypergraph on $[n]$ and $H\subset V^{[k]}$ is a $k$-uniform
hypergraph then
$T(K,H)\subset V^{[n]}$ denotes
the $(K,H)$-{\bf homomorphism set}, where
$(x_1,x_2,\dots,x_n)\in T(K,H)$ if
$1\to x_1, 2\to x_2,\dots, n\to x_n$ defines a homomorphism.
Clearly $|T(K,H)|=hom(K,H)$. For a subset $E\subset [n]$, $|E|=k$
let $P_E:V^{[n]}\to V^{E}$ be the natural projection and
$P_{s_E}:V^{[k]}\to V^{E}$ be the natural bijection associated to a bijective
map $s_E:[k]\to E$.
Then it is easy to check that
$$T(K,H)=\bigcap_{E\in E(K)} P_E^{-1}\left(P_{s_E} (H)\right)\,.$$
Similarly, $T_{ind}(K,H)\subset V^{[n]}$ denotes
the $(K,H)$-{\bf induced homomorphism set}, \\ where
$(x_1,x_2,\dots,x_n)\in T_{ind}(K,H)$ if
$1\to x_1, 2\to x_2,\dots, n\to x_n$ defines an induced homomorphism. Then
$$T_{ind}(K,H)=\bigcap_{E\in E(K)} P_E^{-1}\left(P_{s_E} (H)\right)
\cap \bigcap_{E'\in E(K)^c} P_{E'}^{-1}\left(P_{s_{E'}} (H^c)\right)\,,$$
where $H^c$ denotes the complement of $H$
in the complete hypergraph on the set $V$. 
A simple inclusion-exlusion argument shows that if a hypergraph sequence
$\{H_i\}^\infty_{i=1}$ is convergent, then for any $k$-uniform hypergraph $F$
the sequence $\{t_{ind}(F,H_i)\}^\infty_{i=1}$ is convergent as well.

\subsection{The Removal and the Regularity Lemmas}\label{RRL}
First we state the Removal Lemma.
\begin{theorem}[Hypergraph Removal Lemma] \label{removallemma}
 For every $k$-uniform hypergraph $K$
  and constant $\epsilon>0$ there exists
 a number $\delta=\delta(K,\epsilon)$ such
  that for any $k$-uniform hypergraph $H$ on the node set $X$ with
  $t(K,H)<\delta$ there is a subset $L$ of $E(H)$ with $L\leq
  \epsilon{{|X|}\choose{k}}$ such that $t(K,H\setminus \hat{L})=0$. 
(\cite{Gow}.
\cite{Ish}, \cite{NRS}, \cite{Tao})
\end{theorem}
Now let us turn to the Regularity Lemma.
 Let $X$ be a finite set, then
$K_r(X)\subset X^r$ denotes the complete $r$-uniform hypergraph on $X$.
An $l$-{\bf hyperpartition} $\cH$ is a family of partitions
$K_r(X)=\cup^l_{j=1} P^j_r$, where $P^j_r$
is an $r$-uniform hypergraph, for $1\leq r \leq k$.
We call $\cH$ $\delta$-{\bf equitable}
if for any $1\leq r \leq k$ and $1\leq i< j \leq l$:
$$\frac{||P^i_r|-|P^j_r||}{|K_r(X)|} <\delta\,.$$
An $l$-hyperpartition $\cH$ induces a partition on $K_k(X)$ the following
way.
\begin{itemize}
\item
Two elements $\underline{a},\underline{b}\in K_k(X)$,
$\underline{a}=\{a_1, a_2,\dots, a_k\}$,
$\underline{b}=\{b_1, b_2,\dots, b_k\}$ are equivalent if there exists
a permutation $\sigma\in S_k$ such that for any subset
$A=\{i_1, i_2,\dots,
i_{|A|}\}\subset [k]$, $\{a_{i_1}, a_{i_2},
\dots, a_{i_{|A|}}\}$  and $\{b_{\sigma(i_1)}, b_{\sigma(i_2)},
\dots, b_{\sigma(i_{|A|})}\}$ are both in the same $P^j_{|A|}$ for some
$1\leq j \leq l$.
\end{itemize}
 It is easy to see that this defines an equivalence
relation and thus it results in a partition $\cup^t_{j=1} C_j$
of $K_k(X)$ into {\bf $\cH$-cells}.
A {\bf cylinder intersection} $L\subset K_r(X)$ is an $r$-uniform hypergraph
defined the following way. Let $B_1$, $B_2$,\dots $B_r$ be
$(r-1)$-uniform hypergraphs on $X$, then an $r$-edge $\{a_1, a_2,\dots, a_r\}$
is in $L$ if there exists a permutation $\tau\in S_r$ such that
$$\{a_{\sigma(1)}, a_{\sigma(2)},\dots, a_{\sigma(i-1)}, a_{\sigma(i+1)},
\dots a_{\sigma(r)}\}\in B_i\,\,\mbox{for any $1\leq i \leq r$}\,.$$
As in the graph case, we call an $r$-uniform hypergraph $G$ $\e$-regular if
$$\Big|\frac{|G|}{|K_r(X)|}-\frac{|G\cap L|}{|L|}\Big| \leq \e\,,$$
for each cylinder intersection $L$, where $|L|\geq \e |K_r(X)|\,.$
Now we are ready to state the hypergraph regularity lemma for $k$-uniform
hypergraphs (see \cite{Gow}, \cite{Ish}, \cite{RS}, \cite {RSko}, \cite{Tao}).
\begin{theorem}[Hypergraph regularity lemma]
\label{regularitylemma} Let fix a constant $k>0$. Then
for any $\e>0$ and function $F:\bN\to(0,1)$ there exist
constants $c=c(\e,F)$ and $N_0(\e, F)$ such that if $H$ is
a $k$-uniform hypergraph on a set $X$, $|X|\geq N_0(\e,F)$, then there
exists an $F(l)$-equitable $l$-hyperpartition $\cH$ for some
$1<l\leq c$ such that
\begin{itemize}
\item  Each $P^r_j$ is $F(l)$-regular.
\item $|H\triangle T|\leq \e {{|X|}\choose{k}}\,$
where $T$ is the union of some $\cH$-cells.
\end{itemize}
\end{theorem}
\subsection{Combinatorial Structures} \label{combstruct}

In this subsection
 we introduce some further 
definitions about hyperpartitions. 
 Let $\cH=\{P^j_r\}$ be an $l$-hyperpartition 
on a set $X$ where $1\leq j\leq l$ and $1\leq r\leq k$. 
We shall need the notion of a {\bf directed $\cH$-cell}. 
Let $f:r([k])\mapsto [l]$ be an arbitrary function. 
Then the directed cell with coordinate $f$ is the set of ordered 
$k$-tuples $(x_1,x_2,\dots,x_k)\in X^k$ 
such that $\{x_{i_1},x_{i_2},\dots,x_{i_r}\}\in P^{f(S)}_r$ 
for every set $S=\{i_1,i_2,\dots,i_r\}\in r([k])$.

The symmetric group $S_k$ is acting on $X^k$ by permuting the coordinates and
this action induces an action on the directed $\cH$-cells. 
Note that a $\cH$-cell in the non-directed sense is the union of 
an orbit of a directed $\cH$-cell under the action of $S_k$.

An abstract $(k,l)$-cell is a function $c:r([k])\mapsto [l]$. 
A $(k,l)$-cell system $\mathcal{C}$ is a subset of all possible
$(k,l)$-cells. The symmetric group $S_k$ is acting on $r([k])$ and this
induces 
an action on the $(k,l)$-cells. We say that the system $\mathcal{C}$ is 
symmetric if it is invariant under the action of $S_k$. Such a symmetric
$(k,l)$-system shall be called a {\bf combinatorial structure}.

\noindent
Thus if $\cH$ is an $l$-hyperpartion on $[n]$ and $\mathcal{C}$ is
a combinatorial structure then we can define a $k$-uniform
hypergraph $H(\cH,\mathcal{C},[n])$ the following way.
The hypergraph $H(\cH,\mathcal{C},[n])$ is the union of those $\cH$-cells
in $[n]$ which belong to the coordinates of the combinatorial structure
$\mathcal{C}$.  If $F$ is a 
$k$-uniform hypergraph then we may compute the homomorphism density 
of $F$ in a combinatorial structure $\mathcal{C}$ as follows.
Assume that $V(F)=[n]$ and fix a bijection $s_E:[k]\to E$ for each edge 
of $F$. A function $g:r([n],k)\mapsto [l]$ is
 called a homomorphism of $F$ into $\mathcal{C}$ if for every edge $E$ 
the restriction $g\circ s_E:r[k]\to [l]$ is a $(k,l)$-cell of $\mathcal{C}$.
The homomorphism density $t(F,\mathcal{C})$ is 
the probability that a random map $f:r([n],k)\mapsto [l]$ is a homomorphism.

\subsection{Regularity Lemma as compactness}

In this section we state a new type of regularity lemma together with a 
counting lemma which implies the one stated in the previous section. An 
interesting feature of this regularity lemma is that arbitrarily 
decreasing functions (which are common features in ``strong'' regularity 
lemmas) are replaced by a sequential compactness type statement. 

\begin{theorem}[Hypergraph Sequence Regularity Lemma]\label
{HSRL} For every $\epsilon>0$ and $k$-uniform increasing hypergraph sequence 
$\{H_i\}_{i=1}^\infty$ there is a natural number $l=l(\e,\{H_i\}_{i=1}^\infty)
$ such that there is a subsequence $\{H'_i\}_{i=1}^\infty$ of 
$\{H_i\}_{i=1}^\infty$ 
together with a sequence of $l$-hyperpartitions $\{\cH_i\}_{i=1}^\infty$ 
satisfying the following properties
\begin{enumerate}
\item For every $i$ there is $T_i$ which is the union of some $\cH_i$-cells 
such that $|H'_i\triangle T_i|\leq \e {{|X_i|}\choose{k}}\,$
where $T$ is the union of some $\cH_i$-cells and $X_i$ is the vertex set 
of $H'_i$.
\item The hyperpartition $\cH_i$ is $\delta_i$-equitable and
 $\delta_i$-regular where $\lim_{i\to\infty}\delta_i=0$.
\item Every $T_i$ has the same combinatorial structure $\mathcal{C}$
\item $\lim_{i\to\infty}t(F,T_i)=t(F,\mathcal{C})$ for every $k$-uniform
 hypergraph $F$.
\end{enumerate}
\end{theorem}

Note that the value of $l$ depends on the concrete 
sequence $\{H_i\}_{i=1}^\infty$. To see this one can take a large random 
graph $G$ on $n$ vertices and then take the $i$-fold equitable blowups $G_i$ 
of $G$. The reader can check that in this case (with high probability) $l=n$ 
for any $\epsilon<1/2$.

It is quite natural to interpret Theorem \ref{HSRL} in terms of compactness.

\begin{definition} An increasing hypergraph sequence $\{H_i\}_{i=1}^\infty$ 
is called {\bf strongly convergent} if for every $\epsilon>0$ there 
is a number $l$, hypergraphs $T_i$ on the vertex sets $X_i$ 
of $H_i$ and $l$-hyperpartitions $\cH_i$ on $X_i$ for every $i$ such that
\begin{enumerate}
\item $T_i$ is the union of some $\cH_i$-cells
\item $|H_i\triangle T_i|\leq \e {{|X_i|}\choose{k}}$
\item The hyperpartition $\cH_i$ is $\delta_i$ regular and $\delta_i$ 
equitable where $\lim_{i\to\infty}\delta_i=0$.
\item Every $T_i$ has the same combinatorial structure.
\end{enumerate}
\end{definition}

Using this definition the sequence regularity lemma gets the 
following simple form:

\begin{theorem}[Regularity as Compactness]\label{compa} Every
hypergraph sequence has a strongly convergent subsequence.
\end{theorem}

\subsection{Euclidean hypergraphs}
The goal of this subsection is to generalize the notion of $k$-uniform
hypergraphs and homomorphism densities
to the Euclidean setting in order to define limit objects
for convergent sequences of finite hypergraphs.
Seemingly, the appropriate Euclidean analogue of $k$-uniform hypergraphs
would be just the $S_k$-invariant measurable subsets of $[0,1]^k$. One could
easily define the notion of homomorphisms from finite $k$-hypergraphs
to such Euclidean hypergraphs and even the associated homomorphism densities.
The problem with this simple notion of Euclidean hypergraphs is that they could
serve as limit objects only for very special finite hypergraph sequences.
In order to construct (see Example 1.) limit objects to the
various random construction of convergent hypergraph sequences one needs
a little bit more complicated notion.

\noindent
Let $k>0$ and consider $[0,1]^{2^k-1}=[0,1]^{r([k])}$,
that is the set of points in the form \\
$(x_{A_1},x_{A_2},\dots,x_{A_{2^k-1}})$,
where $A_1,A_2,\dots,A_{2^k-1}$ is a list of the non-empty subsets of $[k]$.
Observe that the symmetry group $S_k$ acts on $[0,1]^{r([k])}$ by
$$\pi((x_{A_1},x_{A_2},\dots,x_{A_{2^k-1}}))=
(x_{\pi^{-1}(A_1)},x_{\pi^{-1}(A_2)},\dots,x_{\pi^{-1}(A_{2^k-1})})\,.$$
We call a measurable $S-k$-invariant subset $\cH\subseteq [0,1]^{2^k-1}$
a $k$-uniform {\bf Euclidean hypergraph}. Now let $K$ be a finite $k$-uniform
hypergraph and let $\Sigma(K)\subseteq
r([n],k)$ be the simplicial complex of $K$ consisting
of the non-empty subsets of the $k$-edges of $K$.
Let $C_1,C_2,\dots,C_{|\Sigma(K)|}$ be a list of the elements of $\Sigma(K)$.
\begin{definition}[Euclidean hypergraph homomorphism]
A map $g: r([n],k)\to [0,1]$ is called a Euclidean
hypergraph homomorphism
from $K$ to $\cH$
if for any edge $E\in E(K)$:
$$(g(s_E(A_1)), g(s_E(A_2)),\dots, g(s_E(A_{2^k-1})))\in\cH\,,$$
where $s_E:[k]\to E$ is a fixed bijection. The induced Euclidean hypergraph
homomorphism is defined accordingly.
\end{definition}
 Note that the notion of
 hypergraph homomorphism does not depend on the choice of $s_E$.
Thus the {\bf Euclidean hypergraph homomorphism set}
$T(K,\cH)\subset [0,1]^{r([n],k)}$
is the set of points \\$(y_{B_1},y_{B_2},\dots, y_{B_{|r([n],k)|}})$ such that
the map $g:\to [0,1], g(B_i)=y_{B_i}$ is a homomorphism.
One can similarly define the 
{\bf Euclidean hypergraph induced homomorphism set}.
We call $\lambda(T(K,\cH))$ the $|\Sigma(K)|$-dimensional Lebesgue-measure
of the homomorphism set the {\bf homomorphism density}. We say that
the hypergraph $\cH$ is the limit of the $k$-uniform hypergraphs
$\{H_n\}^\infty_{n=1}$ if
$$\limn t(K,H_n)=\lambda(T(K,\cH))$$
for any finite $k$-uniform hypergraph $K$.

\vskip 0.1in
\noindent
{\bf Example 1.}
There are many ways to define random $k$-uniform hypergraph
sequences. The most natural one is the random sequence $\{H_n\}^\infty_{n=1}$,
where each edge of the complete hypergraph on $n$-vertices is chosen
with probability $\frac{1}{2}$ to be an edge of $H_n$. Thus for any $k$-uniform
hypergraph $K$, $\limn t(K,H_n)=(\frac{1}{2})^{|E(K)|}$ with probability $1$.
Let us consider the hypergraph
$$\cH=\{(x_{A_1},x_{A_2},\dots,x_{A_{2^k-1}})\in [0,1]^{2^k-1}\,\mid
\,0\leq x_{[k]}\leq \frac{1}{2}\}\,.$$
An easy calculation shows that $\lambda(T(K,\cH))=(\frac{1}{2})^{|E(K)|}$
that is $\cH$ is the limit of a  random hypergraph sequence
$\{H_n\}^\infty_{n=1}$ with
probability $1$.
\vskip 0.1in
\noindent
{\bf Example 2.}
Now we consider a different notion of randomness. Let the random
sequence $\{H'_n\}^\infty_{n=1}$ be constructed the following way.
First choose each $(k-1)$-subset of $[n]$ randomly with probability
$\frac{1}{2}$. Then $E$ will be an edge of $H'_n$ if all its
$(k-1)$-dimensional hyperedges are chosen.
Clearly, $\limn t(K,H'_n)=(\frac{1}{2})^{|K|_{k-1}}$ with probability
$1$, where $|K|_{k-1}$ is the
number of \\ $(k-1)$-hyperedges in $\Sigma(K)$. Now we consider the hypergraph
$$\cH'=\{(x_{A_1},x_{A_2},\dots,x_{A_{2^k-1}})\in [0,1]^{2^k-1}\,\mid
\,0\leq x_{1,2,3,\dots,k-1}\leq \frac{1}{2}\,,
0\leq x_{1,2,3,\dots,k-2,k}\leq \frac{1}{2}, $$$$\dots,
0\leq x_{2,3,\dots,k}\leq \frac{1}{2}\}\,.$$
Then $\lambda(T(K,\cH'))=(\frac{1}{2})^{|K|_{k-1}}$.
Thus $\cH'$  is the limit of a random hypergraph sequence
$\{H'_n\}^\infty_{n=1}$ with
probability $1$.

\vskip 0.1in
\noindent
Now let $K$ be a finite $k$-uniform hypergraph. For any $E\in E(K)$ we fix
a bijection $s_E:[k]\to E$ as above.
Let $L_{s_E}:[0,1]^{r([k])}\to [0,1]^{r([E])}$,
$$L_{s_E}(x_{A_1},x_{A_2},\dots,x_{A_{2^k-1}})=\\
(x_{s_E(A_1)}, x_{s_E(A_2)},\dots, x_{s_E(A_{2^k-1})})$$
be the natural measurable isomorphism associated to the map
$s_E$. Also,
let $L_E:[0,1]^{r([n],k)}\to [0,1]^{r(E)}$ be
the natural projection.
Then for a $k$-uniform Euclidean hypergraph $\cH$ and a finite
$k$-uniform hypergraph $K$ on $n$ vertices
\begin{equation}\label{euclid}
T(K,\cH)=\bigcap_{E\in E(K)} L^{-1}_E(L_{s_E} (\cH))\,.
\end{equation}
Also, 
\begin{equation}\label{inducedeuclid}
T_{ind}(K,\cH)=\bigcap_{E\in E(K)} L^{-1}_E(L_{s_E} (\cH))
\cap \bigcap_{E'\in E(K)^c} L^{-1}_{E'}(L_{s_{E'}} (\cH^c))\,.
\end{equation}
We formulate (\ref{euclid}) in an integral form as well.
Let $W_{\cH}:[0,1]^{r([k])}\to \{0,1\}$ be the characteristic
function of the Euclidean hypergraph $\cH$. We call such an object
a {\bf hypergraphon}.
Then
$$\lambda(T(K,\cH))=\int^1_0 \int^1_0 \dots \int^1_0 \left(
\prod_{E\in E(K)} \Psi_E  \right) dx_{C_1}
dx_{C_2}\dots dx_{C_{\Sigma(K)}}\,,$$
where $\Psi_E$ is the characteristic function of $L^{-1}_E
(L_{s_E}(\cH))$.
Clearly,
$$\Psi_E(x_{C_1},x_{C_2},\dots, x_{C_{\Sigma(K)}})=
W_{\cH}(x_{s_E(A_1)}, x_{s_E(A_2)},\dots, x_{s_E(A_{2^k-1})})\,.$$
Thus, we have the integral formula
$$\lambda(T(K,\cH))=$$
$$=\int_0^1\int_0^1\dots\int_0^1\,\left(\prod_{E\in E(K)}
W_{\cH}(x_{s_E(A_1)}, x_{s_E(A_2)},\dots, x_{s_E(A_{2^k-1})})\right) dx_{C_1}
\dots dx_{C_{\Sigma(K)}}\,. $$
{\bf Remark:}
One can introduce the notion of a {\bf projected hypergraphon} $
\widetilde{W}_{\cH}$
 which
is the projection of a hypergraphon to the first $2^k-2$ coordinates, where
the last coordinate is associated to $[k]$ itself.
That is
$$\widetilde{W}_{\cH}(x_{A_1}, x_{A_2}, \dots, x_{A_{2^k-2}})=
\int_0^1 W_{\cH} (x_{A_1}, x_{A_2}, \dots, x_{A_{2^k-1}}) dx_{A_{2^k-1}}\,.$$
That is $\widetilde{W}_{\cH}$ is a $[0,1]$-valued function which is
symmetric under the induced $S_k$-action of its coordinates.
By the classical Fubini-theorem we obtain that if $\cH$ is the limit
of the hypergraphs $\{H_i\}^\infty_{i=1}$ then
$$lim_{i\to\infty}t(K,H_i)=$$
$$=\int_0^1\int_0^1\dots \int_0^1
\prod_{E\in E(K)}\widetilde{W}_{\cH} (x_{s_E(A_1)}, x_{s_E(A_2)},\dots,
x_{s_E(A_{2^k-2})}) dx_{C_1}
dx_{C_2}\dots dx_{C_{|K|_{k-1}}}\,,$$
where the integration is over the variables associated
to the simplices of dimension
less than $k$.
Note that in the case $k=2$ it is just the graph limit formula of
\cite{LSZ}.

Note that for a combinatorial structure $\mathcal{C}$ one can define a
hypergraphon $W_{\mathcal{C}}\subseteq [0,1]^{2^k-1}$.
Recall that an $l$-box $Z$ in $[0,1]^{2^k-1}$ is 
a product set in the form
$$\left(\frac{i_1}{l}, \frac{i_1+1}{l}\right)\times
\left(\frac{i_2}{l}, \frac{i_2+1}{l}\right) \times \dots \times
\left(\frac{i_{2^k-1}}{l}, \frac{i_{2^k-1}+1}{l}\right)\,.$$
The map $f:r[k]\to [l]$, defined by $f(A_j)=i_j$ is the coordinate
function of the box $Z$. Then $W_{\mathcal{C}}$ is the union of the boxes
corresponding to the coordinates of the combinatorial structure
$\mathcal{C}$. It is easy to check that
$t(F,\mathcal{C})=t(F,W_{\mathcal{C}})$ for any $k$-uniform hypergraph
$F$.

\subsection{$W$-random graphs and Sampling}
\label{sub26}
Let us consider the following natural sampling process 
for $k$-uniform hypergraphs. We pick $n$ 
vertices $v_1,v_2,\dots,v_n$ independently and uniformly at random from the
vertex set $X$ of $H$ and then we create a hypergraph $\mathbb{G}(H,n)$ 
with vertex set $[n]$ such that $\{i_1,i_2,\dots,i_k\}$ is 
an edge in $\mathbb{G}(H,n)$ if and 
only if $\{v_{i_1},v_{i_2},\dots,v_{i_k}\}$ is an edge in $H$.
Thus $\mathbb{G}(H,n)$ is a hypergraph valued random variable. 
The distribution of $\mathbb{G}(H,n)$ can be described in terms of 
the homomorphism densities $t_{\rm ind}(F,H)$ where $|V(F)|\leq n$. 
The probability that we see a fixed hypergraph $F$ on $[n]$ 
in $\mathbb{G}(H,n)$ is equal to $t_{\rm ind}(F,H)$. 

Now we generalize sampling for Euclidean hypergraphs $W\subset
[0,1]^{r([k])}$. Let us introduce a random variable $X_S$ for every set 
$S\in r([n],k)$ which are independent and have uniform
 distribution in $[0,1]$. Then $\{i_1,i_2,\dots,i_k\}$ is an edge in
 $\mathbb{G}(W,n)$ if $W(X_{A_1},X_{A_2},\dots,X_{A_{2^k-1}})=1$ where
 $A_1,A_2,\dots,A_{2^k-1}$ are the non empty subsets of
 $\{i_1,i_2,\dots,i_k\}$.
 This again gives a hypergraph valued random variable on $[n]$ which is 
the infinite analogy of the finite setting.

Another important sampling process from $W$ will be called the 
{\bf hyperpartition sampling}. Assume that 
$\cH=\{P_j^r\}_{1\leq j\leq l,1\leq r\leq k}$ is an 
$l$-hyperpartition on the set $[n]$.
We consider the function $g:r([n],k)\to [l]$,
which is equal to $j$ if and only if $S\in P_j^{|S|}$.
 Now we define a sampling process $\mathbb{G}(W,\cH,n)$ in the 
same way as $\mathbb{G}(W,n)$ with the extra restriction that 
$X_S$ has uniform distribution in the interval $[(g(S)-1)/l,g(S)/l)$.
This sampling process has the property that 
$t_{\rm ind}(F,W)=0$ implies $t_{\rm ind}(F,\mathbb{G}(W,\cH,n))=0$ 
with probability $1$.

\noindent
Finally, we introduce the notion of random coordinate systems.
Let $Z_n$ be the random variable which is a random point in 
$[0,1]^{r([n],k)}$ with uniform distribution. 
In other words $Z_n$ is a $r([n],k)$-tuple of independent random
variables with uniform distribution $\{f_T\}_{T\in r([n],k)}$.
Let $[n]^k_0$ be the set of elements in $[n]^k$ without having
repetitions in their coordinates.
We introduce the random variables
 $\tau^n:[n]^k_0\mapsto [0,1]^{r([k])}$ such that the
component $\tau^n_S(x_1,x_2,\dots,x_k)$ corresponding to an 
element $\{i_1,i_2,\dots,i_t\}=S\in r([k])$ is equal
 to the value of $f_{x_{i_1}, x_{i_2},\dots, x_{i_t}}$.
We call the random variables $\tau^n$ {\bf random coordinate systems}
corresponding to $[n]$. An important property of $(\tau^n)$ is that for a
measurable set $W\subseteq [0,1]^{r([k])}$ the distribution of the
random hypergraph-valued function
$(\tau^n)^{-1}(W)$ is exactly the same as the
distribution of $\mathbb{G}(W,[n])$.

\subsection{Ultraproducts of finite sets}\label{sec1}
First we recall the ultraproduct construction of finite probability
measure spaces (see \cite{Loeb}).
Let $\{X_i\}^\infty_{i=1}$ be finite sets.
We always suppose that $|X_1|<|X_2|<|X_3|<\dots$ Let $\omega$ be a nonprincipal
ultrafilter and $\lim_{\o}:l^\infty(\bN)\to\bR$ be the corresponding
ultralimit. Recall that $\limo$ is a bounded linear functional
such that for any $\epsilon>0$ and $\{a_n\}_{n=1}^\infty\in l^\infty(\bN)$
$$\{ i\in \bN\,\mid\, a_i\in [\limo a_n-\e, \limo a_n +\e]\}\in\omega\,.$$
The ultraproduct of the sets $X_i$ is defined as follows.

\noindent
Let $\widetilde{X}=\prod^\infty_{i=1}X_i$. We say that
$\widetilde{p}=\{p_i\}^\infty_{i=1}, \widetilde{q}=\{q_i\}^\infty_{i=1}\in
\widetilde{X}$ are equivalent,
$\widetilde{p}\sim\widetilde{q}$, if
$$\{i\in \bN\mid p_i=q_i\}\in \omega\,.$$
Define $\xo:=\widetilde{X}/\sim$.
Now let $\cP(X_i)$ denote the Boolean-algebra of subsets of $X_i$, with the
normalized measure $\mu_i(A)=\frac{|A|}{|X_i|}\,.$
Then let $\widetilde{\cP}=\prod^\infty_{i=1}\cP(X_i)$ and
$\cP=\widetilde{\cP}/I$, where $I$ is the ideal of elements
$\{A_i\}^\infty_{i=1}$
such that
$\{i\in \bN\mid A_i=\emptyset\}\in \omega\,.$
Notice that the elements of $\cP$ can be identified with certain subsets
of $\xo$: If
$$\overline{p}=[\{p_i\}^\infty_{i=1}]\in \xo\,\,\mbox{and}\,\, \overline{A}=
[\{A_i\}^\infty_{i=1}]\in \cP$$
then $\overline{p}\in \overline{A}$ if
$\{i\in \bN\mid p_i\in A_i\}\in \omega\,.$
Clearly, if $\overline{A}=
[\{A_i\}^\infty_{i=1}]$, $\overline{B}=
[\{B_i\}^\infty_{i=1}]$ then
\begin{itemize}
\item
$\overline{A}^c=[\{A^c_i\}^\infty_{i=1}]\,,$
\item
$\overline{A}\cup \overline{B}=[\{A_i\cup B_i\}^\infty_{i=1}]\,,$
\item
$\overline{A}\cap \overline{B}=[\{A_i\cap B_i\}^\infty_{i=1}]\,.$
\end{itemize}
That is $\cP$ is a Boolean algebra on $\xo$.
Now let $\muo(\overline{A})=\lim_{\o} \mu_i(A_i)$. Then $\muo:\cP\to\bR$ is
a finitely additive probability measure. We will call $\overline{A}=
[\{A_i\}^\infty_{i=1}]$ the {\bf ultraproduct} of the sets
$\{A_i\}^\infty_{i=1}$.
\begin{definition}
$N\subseteq \xo$ is a {\bf nullset} if for any $\e>0$ there exists
a set $\overline{A_\e}\in\cP$ such that $N\subseteq \overline{A_\e}$
and $\muo(\overline{A_\e})\leq \e$. The set of nullsets is denoted
by $\cN$.
\end{definition}
\begin{proposition}
$\cN$ satisfies the following properties:
\begin{itemize}
\item if $N\in \cN$ and $M\subseteq N$, then $M\in \cN$.
\item If $\{N_k\}^\infty_{k=1}$ are elements of $\cN$ then
$\cup^\infty_{k=1} N_k\in \cN$ as well.
\end{itemize} \end{proposition}
\begin{proof}
The first part is obvious, for the second part we need the following lemma.
\begin{lemma}\label{l6}
If $\{\overline{A_k}\}^\infty_{k=1}$ are elements of $\cP$
and $\lim_{l\to\infty} \muo(\cup^l_{k=1}\overline{A_k})=t$  then
there exists an element $\overline{B}\in\cP$ such that
$\muo(\overline{B})=t$ and $\overline{A_k}\subseteq \overline{B}$
for all $k\in \bN$.
\end{lemma}
\begin{proof}
Let $\overline{B_l}=\cup^l_{k=1}\overline{A_k}$, $\muo(\overline{B_l})=t_l$,
$\lim_{l\to\infty} t_l=t\,.$ Let
$$T_l=\left\{i\in\bN\,\mid\,
|\mu_i(\cup^l_{k=1} A^i_k)-t_l|\leq \frac{1}{2^l}\,\right\}\,,$$
where $\overline{A_k}=[\{A^i_k\}^\infty_{i=1}]\,.$
Observe that $T_l\in \omega$. If $i\in \cap^m_{l=1}T_l$ but
$i\notin T_{m+1}$, then let $C_i=\cup^m_{k=1} A^i_k\,.$
If $i\in T_l$ for all $l\in \bN$, then clearly $\mu_i(\cup^\infty_{k=1}
A^i_k)=t$ and
 we set $C_i:=\cup^\infty_{k=1} A^i_k\,.$
Let $\overline{B}:=[\{C_i\}^\infty_{i=1}]\,.$ Then
$\muo(\overline{B})=t$ and for any $k\in\bN$:
$\overline{A_k}\subseteq \overline{B}$. \qed \end{proof} \vskip 0.2in

\noindent
Now suppose that for any $j\geq 1$, $\overline{A_j}\in\cN$. Let
$\overline{B}^\e_j\in\cP$ such that $\overline{A_j}\subseteq
\overline{B}^\e_j$ and $\muo(\overline{B}^\e_j)<\e\frac{1}{2^j}$.
Then by the previous lemma, there exists $\overline{B}^\e\in\cP$ such that
for any $j\geq 1$
$\overline{B}^\e_j\subseteq \overline{B}^\e$ and $\muo(\overline{B}^\e)\leq\e$.
Since $\cup^\infty_{j=1} \overline{A_j}\subseteq \overline{B}^\e$, our
proposition follows. \qed \end{proof} \vskip 0.2in
\begin{definition}
We call $B\subseteq \xo$ a {\bf measureable set} if there exists
$\widetilde{B}\in \cP$
such that $B\triangle \widetilde{B}\in \cN$.
\end{definition}
\begin{proposition}
The measurable sets form a $\sigma$-algebra $\bo$ and $\muo(B)=
\muo(\widetilde{B})$
defines a probability measure on $\bo$.
\end{proposition}
\begin{proof}
We call two measurable sets $B$ and $B'$ equivalent, $B\cong B'$ if
$B\triangle B'\in \cN$.
Clearly, if $A\cong A'$, $B\cong B'$ then $A^c\cong (A')^c$,
$A\cup B\cong A'\cup B'$, $A\cap B\cong A'\cap B'$. Also if
$A,B\in\cP$ and $A\cong B$, then $\muo(A)=\muo(B)$. That is
the measurable sets form a Boolean algebra with a finitely additive measure.
Hence it is enough to prove that if $\overline{A_k}\in\cP$ are disjoint sets,
then there exists $\overline{A}\in\cP$ such that
$\cup^\infty_{k=1}\overline{A_k}\cong \overline{A}$ and
$\muo(\overline{A})=\sum^\infty_{k=1}\muo(\overline{A_k})\,.$
Note that by Lemma \ref{l6} there exists $\overline{A}\in\cP$ such that
$\muo(\overline{A})=\sum^\infty_{k=1}\muo(\overline{A_k})$ and
$\overline{A_k}\subseteq \overline{A}$ for all $k\geq 1$.
Then for any $j\geq 1$,
$$\overline{A}\backslash \cup^\infty_{k=1}\overline{A_i}\subseteq
 \overline{A}\backslash \cup^j_{k=1}\overline{A_k}\in \cP\,.$$
Since $\lim_{j\to\infty}\muo(\overline{A}\backslash
\cup^j_{k=1}\overline{A_k})=0,
\overline{A}\backslash \cup^\infty_{k=1}\overline{A_k}\in \cN$ thus
$\cup^\infty_{k=1}\overline{A_k}\cong
\overline{A}$.\qed \end{proof} \vskip 0.2in

\noindent
Hence we constructed an atomless probability measure space $(\xo,\bo,\muo)$.
Note that this space is non-separable, that is it is not measurably
equivalent to the interval with the Lebesgue measure.

\subsection{$\sigma$-algebras and the Total Independence
  Theorem}
We fix a natural number $k$ and we denote by $[k]$ the set $\{1,2,\dots,k\}$.
 Let $X_{i,1},X_{i,2},\dots,X_{i,k}$ be $k$ copies of the finite set $X_i$ and
 for a subset
 $A\subseteq\{1,2,\dots,k\}$ let $X_{i,A}$ denote the direct
 product $\bigoplus_{j\in A}X_{i,j}$.
Let $\xo^A$ denote the ultraproduct of the sets $X_{i,A}$, with a Boolean
algebra $\cP_A$.
 There is a natural
 map $p_A:\xok\to \xo^A$ (the projection). Let $\cB_A$ be the
$\sigma$-algebra of measurable subsets in $\xo^A$ as defined in the previous
sections. Define $\sigma(A)$ as $p_A^{-1}(\cB_A)$, the $\sigma$-algebra
of measurable sets depending only on the $A$-coordinates together with
the probability measure $\mu_A$.
For a nonempty subset $A\subseteq [k]$ let $A^*$ denote the
set system $\{B|B\subseteq
A~,~|B|=|A|-1\}$ and let $\sigma(A)^*$ denote the $\sigma$-algebra $\langle
\sigma(B)|B\in A^*\rangle$. An interesting fact is (as it will turn out in
Section \ref{proofcorr}) that $\sigma(A)^*$ is strictly smaller
than $\sigma(A)$. The following figure shows how the lattice of the
various $\sigma$-algebras look like.
\begin{figure}[h]
 \begin{center}
    \includegraphics[width=3.5in]{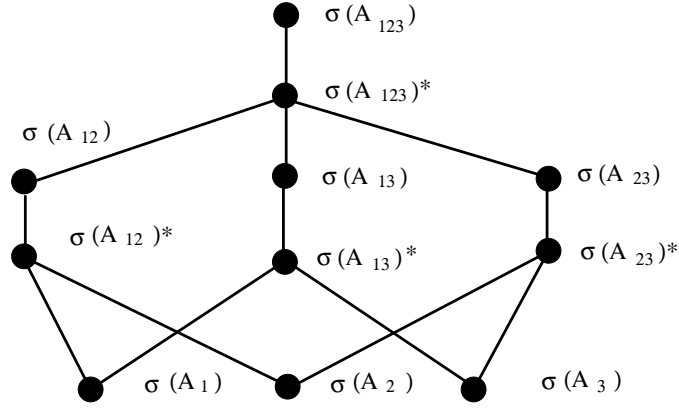}
 \end{center}

  \caption{The $\sigma$-algebras}
  \label{fig-label}
\end{figure}
Recall that if $\cB\subset\cA$ are $\sigma$-algebras on $X$ with  a measure
$\mu$
 and
$g$ is an $\cA$-measurable function on $X$, then $E(g\mid \cB)$ is
the $\cB$-measurable function (unique up to a zero measure perturbation)
 with the property
that
$$\int_Y E(g\mid \cB)\,d\mu=\int_Y g \,d\mu\,,$$
for any $Y\in\cB$ (see Appendix).
If $A\in\cA$ we say that $A$ is independent from the $\sigma$-algebra
$\cB$ if $E(\chi_A\mid\cB)$ is a constant function.
One of the main tool in our paper (the proof will be given in
Section \ref{prooftotal}) is the following theorem.
\begin{theorem}[The Total Independence Theorem]\label{totalindep}
 Let $A_1,A_2,\dots A_r$ be
  a list of distinct nonempty
  subsets of $[k]$, and let $S_1,S_2,\dots,S_r$ be subsets of $\xok$ such
  that $S_i\in\sigma(A_i)$ and $E(S_i|\sigma(A_i)^*)$ is a constant
function for
  every $1\leq i\leq r$. Then
$$\mu(S_1\cap S_2\cap\dots\cap S_r)=\mu(S_1)\mu(S_2)\dots\mu(S_r).$$
\end{theorem}
\section{Correspondance Principles and the proofs of the Removal and
 Regularity Lemmas}
\subsection{The ultraproduct method and the correspondence principles}

The ultraproduct method for hypergraphs relies on various correspondence
principles between the following objects that are infinite variations of
 the concept of a $k$-uniform hypergraph.

\begin{enumerate}

\item An infinite sequence of hypergraphs $H_1,H_2,\dots$ in $\mathcal{H}_k$.
\item The ultraproduct hypergraph ${\bf H}$.
\item A $k$-uniform Euclidean hypergraph $\cH\subseteq [0,1]^{2^k-1}$.

\end{enumerate}

Additionally we will need correspondence principles between homomorphism sets
$$\{T(K,H_i)\}_{i=1}^\infty~~,~~T(K,\bH)~~{\rm and}~~T(K,\cH)$$
for every fixed $k$-uniform hypergraph $K$.
Let $\{H_i\subset X_i^k\}^\infty_{i=1}$ be a sequence
of finite $k$-uniform hypergraphs. Then
the {\bf ultraproduct hypergraph} $\bH=[\{H_i\}^\infty_{i=1}]\subset
\xo^k$ is well-defined. Clearly, $\bH$ is $S_k$-invariant and has no
repetitions in its coordinates.
One can formally define the homomorphism set $T(K,\bH)$ for
any finite $k$-uniform hypergraph $K$ exactly as in Subsection
\ref{homconcomp}. Note that we shall refer to any measurable $S_k$-invariant
set ${\bf P\subset \xo^k}$ without repetitions in its coordinated
a $k$-uniform hypergraph on $\xo$.

The following lemma is a trivial consequence of
the basic properties of the ultraproduct sets.

\begin{lemma}[Homomorphism correspondence I.]\label{homcor1} The
  homomorphism set $T(F,{\bf H})$ is the ultraproduct of the
homomorphism sets $T(F,H_i)$. The
induced  homomorphism set $T_{ind}(F,{\bf H})$ is the ultraproduct of the
homomorphism sets $T_{ind}(F,H_i)$.
\end{lemma}

To state the next theorem we need some notation.
 For an arbitrary set $S$ let $r(S,m)$
denote the set of non-empty subsets of $S$ of size at most $m$ and let $r(S)$
denote $r(S,|S|)$. The symmetric group $S_n$ is acting on $[n]$
and this action induces an action on $r([n],m)$. Furthermore $S_n$ is acting
on $[0,1]^{r([n],m)}$ by permuting the coordinates according to the action on
$r([n],m)$. Let $X,G,G_2$ be sets such that $G_2\subseteq G$. Then
we will denote the projection $X^{G}\mapsto X^{G_2}$ by $P_{G_2}$. If a
function $f$ takes values in $X^{G}$ then for an element $a\in G$ we denote
the corresponding coordinate function by $f_a$ which is the same as the
composition $P_{\{a\}}\circ f$.

\begin{definition}[Separable Realization]\label{desepre} For any
  $k\in\mathbb{N}$ a separable realization is a measure preserving map
$\phi:{\bf X}^k\mapsto [0,1]^{r([k])}$ such that

\begin{enumerate}
\item Any permutation $\pi\in S_k$ commutes with $\phi$ in the
sense that $\phi(\bbx)^\pi=\phi(\bbx^\pi)$.

\item For any $D\in r([k])$ and measurable set $A\subseteq [0,1]$
the set $\phi_D^{-1}(A)$
is in $\sigma(D)$ and is independent from $\sigma(D)^*$.
\end{enumerate}

\end{definition}

\smallskip

\noindent
Note that the fact that $\phi$ commutes with the $S_k$-action means that
$\phi_A(\bbx^\pi)=\phi_{A^{\pi^{-1}}}(\bbx)$ for each $\pi\in S_k$.
The second condition in the previous definition expresses the fact
that the functions $\phi_D$ of a separable realization
depend only on the $D$-coordinates. Also, by Lemma \ref{fremlin}
of the Appendix and the Total Independence Theorem
a separable realization $\phi$ gives a parametrization of $\xo^k$
by $|r([k])|$ coordinates such a way that
$\phi^{-1}$ defines an injective measure algebra homomorphism
from $\cM([0,1]^{r([k])},\cB^k,\lambda^k)$ to a subalgebra of
$\cM(\xo^k,\cB_{[k]},\mu_{[k]})$.
The next theorem is the heart of the hypergraph ultraproduct method.
The proof of it will be discussed in Section \ref{proofcorr}.

\begin{theorem}[Euclidean correspondence]\label{eucc} Let $\mathcal{A}$ be a
  separable sub-$\sigma$-algebra of $\sigma_{[k]}$ on ${\bf X}^k$.
Then there is a separable realization $\phi:{\bf X}^k\mapsto [0,1]^{r([k])}$
such that for every $A\in\mathcal{A}$ there is a measurable set
$B\subseteq[0,1]^{2^k-1}$ with $\mu(\phi^{-1}(B)\triangle A)=0$.
\end{theorem}

\begin{corollary}\label{eucc2} Let $\bE$ be an $S_k$-invariant measurable
  subset of $\xo^k$. Then there is a separable realization $\phi$ and
$S_k$-invariant measurable set $W\subseteq [0,1]^{r([k])}$ such that
$\mu(\phi^{-1}(W)\triangle \bE)=0$.
\end{corollary}

The following definition and lemma will
be needed to state the main correspondence between homomorphism sets.

\begin{definition}[Lifting] Let $\phi:{\bf X}^k\mapsto [0,1]^{r([k])}$ be a
  separable realization and let
$n\geq k$ be an arbitrary natural number. Then a measure preserving map
$\psi:{\bf X}^n\mapsto[0,1]^{r([n],k)}$ is called a degree $n$ {\bf lifting} of
  $\phi$ if $P_{r([k])}\circ\psi$ is equal to $\phi\circ P_{[k]}$ on $\xo^n$
and $\psi(\bbx)^\pi=\psi(\bbx^\pi)$ for all permutations $\pi\in S_n$.
\end{definition}
\begin{lemma}[Lifting exists]\label{lifting} Let $\phi:{\bf X}^k\mapsto
  [0,1]^{r([k])}$ be a separable
realization and let $n\geq k$ be an arbitrary natural number. Then there
exists a degree $n$
lifting  $\psi$ of $\phi$.
\end{lemma}

\begin{proof} Let $A\in r([n],k)$ be an arbitrary set with $t$ elements and
  let $\pi\in S_n$ be a
permutation such that $A^\pi=[t]$. We define $\psi_A({\bf x})$ to be
$\phi_{[t]}(P_{[k]}({\bf x}^\pi))$.
Using the fact that $\phi$ commutes with the $S_k$ action we obtain that
$\phi_A\circ P_{[k]}=\psi_A$ for every $A\in r([k])$.
Now if $\pi_2$ is an arbitrary permutation from $S_n$ then the $A$-coordinate
of $\psi({\bf x})^{\pi_2}$ is the $A^{\pi_2^{-1}}$-coordinate of $\psi({\bf
  x})$ which is the $A$-coordinate of $\psi({\bf x}^{\pi_2})$.
This proves that $\psi$ commutes with $S_n$. It remains to show that $\psi$ is
measure preserving. The coordinate functions $\psi_A$
are constructed in a way which guarantees that they are measure
preserving. Let $I_A\subseteq [0,1]$ be intervals of length $l_A$
for every $A\in r([n],k)$ and let $$W=\prod_{A\in r([n],k)} I_A$$ be
their direct product. Since every measurable set
in $[0,1]^{r([n],k)}$ can be approximated by the disjoint union of such cubes
it is enough to check that $\psi^{-1}$
preserves the measure of such a set $W$. The preimage $\psi^{-1}(W)$ is the
intersection of the preimages $\psi_A^{-1}(I_A)$
which are in $\sigma(A)$ and are independent from $\sigma(A)^*$. Now the
Total Independence Theorem
completes the proof.
\qed\end{proof}

\begin{lemma}[Homomorphism Correspondence II.]\label{homcor2} Let $W\subseteq
  [0,1]^{r([k])}$ be an $S_k$-invariant measurable set and
let $\bE$ be the preimage of $W$ under some separable realization
$\phi$. Then for an arbitrary
finite hypergraph $K$
$$\psi^{-1}(T(K,W))=T(K,\bE)\,,$$ where $\psi$ is a $|K|$ degree lifting
of $\phi$. Similarly, 
$$\psi^{-1}(T_{ind}(K,W))=T_{ind}(K,\bE)\,,$$
\end{lemma}

\begin{proof} Assume that the vertex set of $K$ is defined on $[n]$ and that
  the edges of $K$ are \\ $\pi_1([k]),\pi_2([k]),
\dots,\pi_t([k])$ for some permutations $\pi_1,\pi_2,\dots,\pi_t$ in
$S_n$. Let $\bE_2\subset \xo^n$ be the preimage of $\bE$
under the projection $P_{[k]}$ and let $W_2\subset [0,1]^{r([n])}$ be the
preimage of $W$ under
the projection $P_{r([k])}$. By definition we have that
$$T(K,\bE)=\bigcap_{i=1}^t \bE_2^{\pi_i}$$ and
$$T(K,W)=\bigcap_{i=1}^t W_2^{\pi_i}.$$
Since $\psi$ is a lifting of $\phi$ the first lifting property shows that
$\psi^{-1}(W_2)=\bE_2$. Furthermore since $\psi$ commutes
with the elements of $S_n$ we get that $\psi^{-1}(W_2^\pi)=\bE_2^{\pi}$ for
every $\pi\in S_n$.
This completes the proof.
\qed\end{proof}

\subsection{The proof of the Hypergraph Removal lemma}

\begin{lemma}[Infinite removal lemma] Let ${\bf H}$ be the ultraproduct of
  the $k$-uniform hypergraphs $H_1,H_2,\dots$ and
let $F$ be a finite $k$-uniform hypergraph such that $T(F,{\bf H})$ has
measure $0$. Then there is a $0$-measure $S_k$-invariant subset ${\bf I}$
 of ${\bf H}$ such that $T(F,{\bf H}
\setminus{\bf I})$ is empty.
\end{lemma}

\begin{proof} We use Corollary \ref{eucc2} for the set $\bH$ and
we get a separable realization $\phi$ and a measurable set $W\subseteq
[0,1]^{r([k])}$ satisfying the statement of the corollary.
Let $D$ denote the set density points in $W$. Lebesgue's density theorem
says that $W\setminus D$ has measure $0$. Furthermore $D$ will remain
symmetric under
the action of the symmetric group on $[0,1]^{r([k])}$.
Let ${\bf D}$ be the preimage of $D$ under the map $\phi$. Using the
first property in
definition \ref{desepre} we obtain that ${\bf D}$ is $S_k$-
invariant. Furthermore the
measure of ${\bf H}\triangle {\bf D}$ is $0$.

Now let $F$ be a $k$-uniform hypergraph on the vertex set $[n]$ and let
$\psi$ be a
degree $n$ lifting of $\phi$. Lemma \ref{homcor2} shows that $T(F,{\bf D})$ is
the preimage of $T(F,D)$ under $\psi^{-1}$. On the other hand $T(F,D)$ is
the intersection of finitely
many sets consisting only of density points. This show that $T(F,D)$ and
thus $T(F,{\bf D})$ is either empty or has
positive measure. This means that the set ${\bf I}={\bf H}\setminus {\bf D}$
satisfies the
required condition.
\qed\end{proof}


\noindent{\bf Proof of the hypergraph removal lemma.} We proceed by
contradiction.
Let $K$ be a fixed hypergraph
and
  $\epsilon>0$ be a fixed number for which the theorem fails. This means that
  there is a sequence of hypergraphs $H_i$ on the sets $X_i$ such that
  $lim_{i\to\infty}t(K,H_i)=0$ but in each $H_i$ there is no set $L$ with the
  required property. Again let
  $\bH\subseteq\xo^k$ denote the ultraproduct hypergraph. Then
   $\mu(T(K,\bH))=\limo t(K,H_i)=0$ and thus by the previous lemma
  there is a
  zero measure $S_k$-invariant
  set $\bI\subseteq\xo^k$ such that $T(K,\bH\setminus
  \bI)=\emptyset$. By the definition of nullsets, for any $\e_1>0$ there exists
 an ultralimit set $\bJ\subset \xo^k$ such that $\bI\subset \bJ$ and
$\mu(\bJ)<
 \e_1$. We can suppose that $\bJ$ is $S_k$-invariant as well. Let
 $[\{J_i\}_{i=1}^\infty]=\bJ$, then
for $\omega$-almost all $i$, $J_i$ is $S_k$-
 invariant, $|J_i|\leq \e_1|X_i|^k$ and $T(K,H_i\backslash L_i)=\emptyset$,
 where $L_i$ is the set of edges $\{x_1,x_2,\dots, x_k\}$ such that
 $(x_1,x_2,\dots,x_k)\in J_i$. Clearly, $|L_i|\leq |J_i|$, hence if
$\e_1$ is small enough then $|L_i|\leq \epsilon{{|X_i|}\choose{k}}$
leading to a contradiction.

\subsection{The existence of the Hypergraph Limit Object}\label{hypgr}

\begin{proposition}\label{lemma31} Let $\{H_i\}_{i=1}^{\infty}$ be a
sequence of
  $k$-uniform hypergraphs and let $\bH$ be their ultraproduct hypergraph.
 Assume furthermore that $\phi:\xo^k\mapsto [0,1]^{r([k])}$ is a separable
 realization such that there is an $S_k$-invariant measurable
set $\cH\subseteq [0,1]^{r([k])}$ with
$\mu(\phi^{-1}(\cH)\triangle\bH)=0$. Then for every $k$-uniform
hypergraph $K$ we have that
$$\limo~ t(K,H_i)=t(K,\cH).$$
\end{proposition}

\begin{proof} Let $K$ be a $k$ uniform hypergraph on $n$ vertices and let
  $\psi$ be a degree $n$ lifting of $\phi$.
Lemma \ref{homcor1} implies that $t(K,\bH)=\lim_{\omega} T(K,H_i)$
furthermore, using that $\psi$ is measure preserving,
lemma \ref{homcor2} implies that $t(K,\bH)=t(K,\cH)$.
\qed\end{proof}

The following theorem is an immediate corollary of the previous one.

\begin{theorem} [Existence of the limit object] \label{lobtetel}
If $\{H_i\}_{i=1}^{\infty}$ is a convergent sequence of
  $k$-uniform hypergraphs then there exists a Euclidean
  hypergraph $\cH\subset [0,1]^{r([k])}$
 such that $lim_{i\to\infty}t(K,H_i)=t(K,\cH)$ for every
  $k$-uniform hypergraph $K$.
\end{theorem}

\medskip

\subsection{The proof of the Hypergraph Regularity Lemma}

Suppose that the theorem does not hold for some $\e>0$ and
$F:\bN\to (0,1)$. That is there exists a sequence of $k$-uniform
hypergraphs $H_i$ without having $F(j)$-equitable
$j$-hyperpartitions for any $1<j\leq i$ satisfying the conditions of
our theorem. Let us consider their ultraproduct
$\bH\subset \xo^k$. Similarly to the
proof of the Removal Lemma we formulate an infinite version of the
Regularity Lemma as well.

\noindent Let $K_r(\xo)$ denote the complete $r$-uniform hypergraph
on $X$, that is the set of points \\ $(x_1,x_2,\dots,x_r)\in \xo^r$
such that $x_i\neq x_j$ if $i\neq j$. Clearly $K_r(\xo)\subset
\xo^r$ is measurable and $\mu_{[r]}(K_r(\xo))=1\,.$ An $r$-uniform
hypergraph on $\xo$ is an $S_r$-invariant measurable subset of
$K_r(\xo)$. An $l$-hyperpartition $\wch$ is a family of partitions
$K_r(\xo)=\cup^l_{j=1}{\bf  P^j_r}$, where ${\bf P^j_r}$ is an
$r$-uniform hypergraph for $1\leq r \leq k$. Again, an
$l$-hyperpartition induces a partition of $K_k(\xo)$ into
$\wch$-cells exactly the same way as in the finite case. It is easy
to see that each $\wch$-cell is measurable.

\begin{proposition} [Hypergraph Regularity Lemma, infinite version]
For any $\e>0$, there exists a $0$-equitable
 $l$-hyperpartition (where $l$ depends on
$\bH$) $\wch$ such that
\begin{itemize}
\item Each ${\bf P^j_r}$ is independent from $\sigma([r])^*$.
\item $\mu_{[k]}({\bf H}\triangle T)\leq \e$, where $T$ is a union of some
$\wch$-cells. \end{itemize}
\end{proposition}

\begin{proof} Let $\phi$ be a separable realization for $\bH$ that is such a
  $\phi$ that there exists an $S_k$-invariant subset
$Q\subseteq [0,1]^{2^k-1}$ such that
$\mu_{[k]}(\phi^{-1}(Q)\triangle \bH)=0$. Since $Q$ is a
Lebesgue-measurable set, there exists some $l>0$ such that
$Vol_{2^k-1}(Q\triangle Z)<\e$, where $Z$ is a union of $l$-boxes
(see Subsection \ref{combstruct}).

By the
usual symmetrization argument we may suppose that the set $Z$ is
invariant under the $S_k$-action on the $l$-boxes.
For each $1\leq r \leq k$ we consider
the partition $\xo^r=\cup_{j=1}^l{\bf  P^j_r}$, where ${\bf P^j_r}=
\phi^{-1}_{[r]}(\frac{j-1}{l}, \frac{j}{l})\,.$ We call the resulting
$l$-hyperpartition $\widetilde{\cH}$. Note that by the
$S_r$-invariance of the separable realization each ${\bf P^j_r}$ is
an $r$-uniform hypergraph and also ${\bf P^j_r}$ is
independent from $\sigma([r])^*$.

Now we show that ${\bf C}$ is an $\wch$-cell if and only if ${\bf
C}=\phi^{-1}(\cup_{\pi\in S_k} \pi(D))$, where $D$ is an $l$-box in
$[0,1]^{2^k-1}$. By definition ${\bf a}=(a_1,a_2,\dots,a_k)\in \xo^k$ and
${\bf b}=(b_1,b_2,\dots,b_k)\in \xo^k$ are in the same $\wch$-cell if and
only if there exists $\pi\in S_k$ such that
$(a_{i_1},a_{i_2},\dots, a_{i_{|A|}})$ and
$(b_{i_{\pi(1)}},b_{i_{\pi(2)}}\dots,b_{i_{\pi(|A|)}} )$ are in
the same ${\bf  P^j_r}$ for any $A\subseteq [k]$. That is $\phi({\bf a})$ and
$\phi({\bf b}^\pi)=
(\phi({\bf b}))^\pi$ are in
the same $l$-box.

\vskip 0.1in \noindent Since $Z$ is a union of $S_k$-orbits of $l$-boxes the
set $T=\phi^{-1}(Z)$
is the union of $\widetilde{\cH}$ cells. Using that $\phi$
is measure preserving the proof is complete.
\qed\end{proof}

\vskip 0.2in

\noindent Now we return to the proof of the Hypergraph Regularity
Lemma. First pick an $r$-hypergraph ${\bf \wp^j_r}$ on $\xo$ such
that $\mu_{[r]}({\bf \wp^j_r}\triangle{\bf  P^j_r})=0$, ${\bf
\wp^j_r}\in \cP_{[r]}$ and $\cup_{j=1}^l {\bf \wp^j_r}=K_r(\xo)$.
Let $[\{P^j_{r,i}\}^\infty_{i=1}]={\bf \wp^j_r}\,.$ Then for
$\omega$-almost all indices $\cup_{j=1}^l P^j_{r,i}= K_r(X_i)$ is an
$F(l)$-equitable $l$-partition and $|H_i\triangle \cup^q_{m=1}
C^i_m|<\e{{|X|}\choose{k}}$ for the induced $\cH$-cell approximation. Here
$\cup^q_{m=1} {\bf\widetilde{C}_m}$ is the $\wch$-cell approximation
with respect to the $l$-hyperpartitions $\cup^l_{j=1} {\bf
\wp^j_r}=K_r(\xo)$ and $[\{C^i_m\}^\infty_{i=1}]={\bf
\widetilde{C}_m}$.

\noindent The only thing remained to be proved is that for
$\omega$-almost all indices $i$ the resulting $l$-hyperpartitions
are $F(l)$-regular. If it does not hold then there exists $1\leq r
\leq k$ and $1\leq j \leq l$ such that for almost all $i$ there
exists a cylinder intersection $W_i\subset K_r(X_i)$, $|W_i|\geq
 F(l)|K_r(X_i)|$, such that
\begin{equation} \label{egyenlet}
\left|\frac{|P^j_{r,i}|}{|K_r(X_i)|}-\frac{|P^j_{r,i}\cap W_i|}
{|W_i|}\right|>F(l)\,.
\end{equation}
Let ${\bf W }=[\{W_i\}^\infty_{i=1}]\,.$ Then ${\bf W}\in
\sigma([r])^*$. Hence ${\bf \wp^j_r}$ and ${\bf
W}$ are independent sets. However, by (\ref{egyenlet})
$$\mu_{[r]}({\bf \wp^j_r})\mu_{[r]}({\bf W})\neq \mu_{[r]}({\bf \wp^j_r}
\cap {\bf W})\,,$$ leading to a contradiction. \qed

\subsection{The proof of the Hypergraph Sequence Regularity Lemma}
Let us consider the ultralimit $\bH$ of the hypergraph sequence
$\{H_i\}^\infty_{i=1}$ as in the proof of the regularity lemma
together with the $l$-hyperpartition $\wch$ given by the partition
$\cup^l_{j=1} {\bf\wp^j_r}=K_r(\xo)$, where 
$[\{P^j_{r,i}\}^\infty_{i=1}]={\bf \wp^j_r}\,.$
If $s\geq 1$, then for $\omega$-almost all indices
\begin{itemize}
\item $\cup_{j=1}^l P^j_{r,i}= K_r(X_i)$ is an $\frac{1}{s}$-equitable
$\frac{1}{s}$-regular partition
\item $|H_i\triangle \cup^q_{m=1}
C^i_m|<\e {{|X_i|}\choose{k}}$.
\item $T_i$ has combinatorial structure $\mathcal{C}$, where
$T_i=\cup^q_{m=1} C^i_m$.
\end{itemize}
Also, by Lemma \ref{homcor1} and Lemma \ref{homcor2}
$$\limo t(F,T_i)=t(F,(\cup^q_{m=1}
{\bf\widetilde{C}_m}))=t(F,\mathcal{C})\,.$$
Thus for $\omega$-almost all $i$, $|t(F,T_i)-t(F,\mathcal{C})|<\frac{1}{s}\,.$
Therefore we can pick a subsequence $H'_i$ satisfying the four conditions
of the Hypergraph Sequence Regularity Lemma. \qed
\subsection{Testability of Hereditary Properties}

We omit here the definition of Property Testing but we state a theorem
 which is equivalent with the statement that hereditary hypergraph properties 
are testable.

\begin{theorem}\label{testable}
 Let $\mathcal{F}$ be a family of $k$-uniform hypergraphs. Then 
for every $\epsilon>0$ there is a $\delta=\delta(\epsilon,\mathcal{F})>0$ and 
a natural number $n=n(\epsilon,\mathcal{F})$ such that if $H$ satisfies 
$t_{\rm ind}(F,H)\leq\delta$ for every $F\in\mathcal{F}$ 
with $V(F)\leq n$ then there is a hypergraph $H'$ on
 the vertex set $X$ of $H$ with 
$|H\triangle H'|\leq\e {{|X|}\choose{k}}$ 
such that $t^0_{\rm ind}(F,H')=0$ for every $F\in\mathcal{F}$.
(see also \cite{RStest},\cite{Austin} and \cite{Austintao})
\end{theorem}

\begin{proof} We proceed by contradiction. 
Assume that there is a sequence $\{H_i\}_{i=1}^\infty$ and $\e >0$ 
such that $\lim_{i\to\infty}t_{\rm ind}(F,H_i)=0$ for every $F\in\mathcal{F}$,
however no member of the sequence can be modified in the way guaranteed by 
the theorem. Let us repeat the construction used in the proof
of the Regularity Lemma again. 
Let $\bH$ be the ultralimit hypergraph of
$\{H_i\}^\infty_{i=1}$. We use Corollary \ref{eucc2} 
for the set $\bH$ in order to
obtain a separable realization $\phi$ and a measurable set $W\subseteq
[0,1]^{r([k])}$ satisfying the statement of the corollary. 
Then $t_{\rm ind}(F,W)=\limo t_{\rm ind}(F,H_i)=0$ 
for every $F\in\mathcal{F}$.

\noindent
Thus there is an $l$-step Euclidean hypergraph (a union of $l$-boxes)
$W'$ such that $Vol(W\triangle W')\leq\e/4$. Let ${\bf Q}$ be the preimage 
of $W'$ under $\phi$. Denote by $\mathcal{C}$ the combinatorial structure
of $W'$.
As in the proof of the regularity lemma 
for each $1\leq r \leq k$ we consider
the partition $\xo^r=\cup_{j=1}^l{\bf  P^j_r}$, where ${\bf P^j_r}=
\phi^{-1}_{[r]}(\frac{j-1}{l}, \frac{j}{l})\,.$ We call the resulting
$l$-hyperpartition $\widetilde{\cH}$. The set ${\bf Q}$ is the union of some
cells in $\widetilde{\cH}$. Again,  we modify the sets ${\bf P^j_r}$
to obtain the sets $[\{P^j_{r,i}\}^\infty_{i=1}]={\bf \wp^j_r}\,.$
Consider the resulting $l$-hyperpartitions $\cH_i$ on $X_i$ and for every 
$i$ denote 
the union of $\cH_i$-cells with coordinates in $\mathcal{C}$ by $Q_i$.
That is 
$Q_i=\mathbb{G}(W',\cH_i,[n_i])$, where $[n_i]$ is the vertex set of $H_i$ .
Note that $\mathbb{G}(W',\cH_i,[n_i])$
 is a random hypergraph, nevertheless it
always takes the same value.
Then of course, $\mu({\bf Q'\triangle Q})=0$
where ${\bf Q'}$ is the ultralimit of the hypergraphs $\{Q_i\}^\infty_{i=1}$.

\noindent
Now we consider the random hypergraph model $G_i=\mathbb{G}(W,\cH_i,[n_i])$.
For an ordered set $S=(i_1,i_2,\dots,i_k)\in [n_i]^k$ let $Y_S$ denote the
random variable which takes $1$ if $S$ is in 
$G_i\triangle Q_i$ and takes $0$ elsewhere.
One can easily see that the expected value of
$Y_S$ is $l^{2^k-1}{\rm Vol}((W'\cap B)\triangle(W\cap B))$ 
where $B$ is the box representing the coordinate of the directed cell 
containing $S$.
This shows that
 $$E(|G_i\triangle Q_i|)=\sum_S E(Y_S)=\sum_C 
|C|l^{2^k-1}{\rm Vol}((W'\cap B(C))
\triangle (W\cap B(C)))\,,$$ where $C$ runs through the directed 
cells of $\cH_i$ and $B(C)$ is the box in $[0,1]^{2^k-1}$ corresponding to the
coordinate of $C$.

\noindent
Observe that $\limo |C_i^f|/(n_i)^k=l^{2^k-1}$ 
where $C_i^f$ is the cell in $\cH_i$ corresponding to
the coordinate $f$ . Indeed, the ultralimit of $\{C_i^f\}^\infty_{i=1}$
is a cell in the $l$-hyperpartition of ${\bf X}$.
 That is $$\limo
\frac{E(|G_i\triangle Q_i|)} {n_i^k}={\rm Vol}(W'\cap W)\leq \e/4.$$
 On the other hand we know 
that 
$\limo \frac{|Q_i\triangle H_i|} {n_i^k}=|{\bf Q}\triangle {\bf H}|\leq \e/4$.

\noindent
Consequently, $\limo \frac{E(|G_i\triangle H_i|)} {n_i^k}\leq \e/2\,.$
Note that by probability $1$, $t_{ind}(F,G_i)=0$ for any $F\in \cF$.
That is there exists a hypergraph $H'_i$ which is a value of the
hypergraph valued random variable $G_i$ such that
\begin{itemize}
\item $t_{ind}(F, H'_i)=0.$
\item $\limo \frac{|H'_i\triangle H_i|} {n_i^k}< \e\,.$
\end{itemize}
This leads to a contradiction. \qed

\end{proof}

\section{Uniqueness results and metrics}

\subsection{Distances of hypergraphs and hypergraphons}

Let $U$ and $W$ be two measurable sets in $[0,1]^{r([k])}$. 
The distance $d_1(U,W)$ is defined as the measure of their symmetric 
difference $U\triangle W$.
Let $F$ be a $k$ uniform hypergraph. It is clear from the definitions that
$$|t(F,U)-t(F,W)|\leq |E(F)|d_1(U,W).$$
We can also introduce a distance using subhypergraph-densities.

Let $\delta=\delta_w(U,W)$ denote the smallest number such that
$$|t(F,U)-t(F,W)|\leq |E(F)|\delta,\quad\mbox{for any $F$}.$$
Clearly, $\delta_w(U,W)\leq d_1(U,W)$. It is easy to see that 
$\delta_w$ satisfy the triangle inequality. On the other hand $\delta_w$ 
is only a pseudometric since (as we will see) there are different sets 
$U$ and $W$ with $\delta_w(U,W)=0$. Our goal is to understand which two 
functions have distance $0$ in the pseudometric $\delta_w$.

For every set $S\in r([k])$ we denote by $\mathcal{A}_S$ the $\sigma$-algebra
generated by the projection \\ $[0,1]^{r([k])}\mapsto[0,1]^{r(S)}$. 
 Let $\mathcal{A}_S^*$ denote the $\sigma$-algebra generated by all 
the algebras $\mathcal{A}_T$ where $T$ is a proper subset of $S$.
for every $S\in r([k])$
We say that a measurable map $\phi:[0,1]^{r([k])}\mapsto[0,1]^{r([k])}$ is 
{\bf structure preserving} if

\begin{enumerate}
\item $\phi$ is measure preserving.
\item $\phi^{-1}(\mathcal{A}_S)\subseteq\mathcal{A}_S$.
\item The sets $\phi_S^{-1}(I)$ are independent 
from $\mathcal{A}_S^*$ for every measurable set $I\subseteq [0,1]$.
\item $\phi\circ\pi=\pi\circ\phi$ for every permutation in $S_k$.
\end{enumerate}

The following lemma shows that structure preserving maps do not 
change the homomorphism densities in hypergraphons.

\begin{lemma} \label{lem41} For any structure preserving map $\phi$ 
we have that $\delta_w(U,\phi^{-1}(U))=0$.
\end{lemma}
\begin{proof}
We need to prove that for any finite $k$-uniform hypergraph $F$
$$t(F,U)=t(F,\phi^{-1}(U))\,.$$
Mimicking the proof of Lemma \ref{lifting} we can easily
see that there exists a map
$\hat{\phi}: [0,1]^{r([n],k)}\to [0,1]^{r([n],k)}$ such that
$\hat{\phi}$ commutes with the $S_n$-action and
$$\phi\circ L_{[k]}=L_{[k]}\circ \hat{\phi}\,$$
where $ L_{[k]}$ is the projection to the $[k]$-coordinates.
Therefore, we have the following formula for the homomorphism sets:
$$\hat{\phi}^{-1}\left(\cap_{E\in E(F)} L^{-1}_E(L_{s_E} (U))\right)=
\cap_{E\in E(F)} L^{-1}_E(L_{s_E}(\phi^{-1}(U))\,.$$
Hence the lemma follows. \qed \end{proof}
\begin{definition}\label{strequ}
 A structure preserving map $\psi:[0,1]^{r([k])}\mapsto
  [0,1]^{r([k])}$ is called a structure preserving equivalence if 
there is a structure preserving map $\phi$ such that both 
$\psi\circ\phi$ and $\phi\circ\psi$ 
are equivalent  to the identity map on $[0,1]^{r([k])}$ (recall that
equivalence means that two maps define the same measure algebra homomorphism).
\end{definition}
Now we introduce the pseudodistance $\delta_1$ by the formula
$$\delta_1(U,W)=\inf_{\phi,\psi}d_1(\phi^{-1}(U),\psi^{-1}(W))\,,$$
where $\phi$ and $\psi$ run through all the structure preserving 
transformations.
We will prove the following uniqueness theorem (see \cite{Lovaszunique}
for the graph case)

\begin{theorem}[Uniqueness I.] \label{uni1}$\delta_w(U,W)=0$ if and only if 
there are two structure preserving measurable maps
$\phi,\psi:[0,1]^{r([k])}\mapsto [0,1]^{r([k])}$
such that the measure of $\phi^{-1}(U)\triangle\psi^{-1}(W)$ is zero.
\end{theorem}

\begin{theorem}[Uniqueness II.] \label{uni2} $\delta_w(U,W)=0$ if and only 
if $\delta_1(U,W)=0$.
\end{theorem}
\subsection{Technical Lemmas}
First we prove a simple real analysis lemma.
\begin{lemma} \label{lemmafontos}
Let $Y\subseteq [0,1]^n$ be a measurable
 set independent from the $\sigma$-algebra
$\cA_{n-1}$ generated by the projection onto
the first $(n-1)$-coordinates. Then there exist measurable
subsets $X_k\subseteq [0,1]^n$
in the form 
$$X_k=(A^k_1\times B^k_1)\cup(A^k_2\times B^k_2)\cup\dots\cup 
(A^k_{n_k}\times B^k_{n_k})\,$$ such that
$\lim_{k\to\infty} Vol(X_k\triangle Y)=0$,
where
$A^k_1\cup A^k_2\cup\dots\cup A^k_{n_k}$ is
a measurable partition of $[0,1]^{n-1}$ and $\lambda(B^k_1)=
\lambda(B^k_2)=\dots=\lambda(B^k_{n_k})=Vol(X_k)$.
Obviously, the sets $X_k$ are all independent from $\cA_{n-1}$.
\end{lemma}
\proof
Fix a real number $\epsilon>0$. Let $H$ be a union of $l$-boxes
in $[0,1]^n$ such that $l>\frac{1}{1000\epsilon^2}$ and
$Vol(H\triangle Y)<\frac{\epsilon}{1000}\,.$
By Fubini's Theorem, for almost all $z\in[0,1]^{n-1}$,
$\lambda(A_z^Y)=Vol(Y)$, where
$$A^Y_z=\{t\in [0,1], \,(z,t)\in Y\}\,.$$
For each $l$-box $T$ in $[0,1]^{n-1}$ let 
$$H_T=\{s\in [0,1]\,,T\times s\in H\}\,.$$

\begin{lemma}
The number of $l$-boxes in $[0,1]^{n-1}$ for which
$|\lambda(H_T)-Vol(Y)|>\frac{\epsilon}{10}$ is less than
$\frac{\epsilon}{10} l^{n-1}$.
\end{lemma}
\proof
By Fubini's Theorem,
$$\sum_T |\lambda(H_T)-Vol(Y)|\leq Vol(H\triangle Y)\,.$$
Hence the lemma follows. \qed

\bigskip
\noindent
Now the set $X_\epsilon$ is constructed the following way.
Pick an integer $m$ such that 
$|\frac{m}{l}-Vol(Y)|<\frac{\epsilon}{10}\,.$
If for an $l$-box $T$ $|\lambda(H_T)-Vol(Y)|<\frac{\epsilon}{10}$
then add or delete less than $\frac{\epsilon}{10}l$
 $l$-boxes of $H$ above $T$ to obtain exactly $m$ boxes. On the other hand
if $|\lambda(H_T)-Vol(Y)|\geq\frac{\epsilon}{10}$, then
just pick $m$ arbitrary boxes above $T$. Then $X_\epsilon$
is in the right form and $Vol(X_\epsilon\triangle Y)\to 0$ as $\epsilon\to 0$.
\qed

\bigskip
\noindent
The following lemma establishes the functorality of
separable realizations and structure preserving maps.
\begin{lemma} \label{functor}
Let $\phi:\xo^k\to [0,1]^{r([k])}$ be
a separable realization and $\rho: [0,1]^{r([k])}\to [0,1]^{r([k])}$
be a structure preserving map. Then $\rho\circ\phi$ is a separable
realization as well. Similarly the compositions of two
structure preserving maps, or the inverse of a structure preserving
equivalence is a structure preserving map.
\end{lemma}
\proof
For the first part is enough to prove that
if $M\subseteq [0,1]^{r([k])}$, $M\in \cA_S$ for some $S\subseteq [k]$
 such that
$M$ is independent from $\cA^*_S$ then $\phi^{-1}(M)$ is independent
from $\sigma(S)^*$.

\noindent
First suppose that $M$ is in block-form that is
$$M=\cup^n_{i=1}(A_i\cap B_i)\,,$$
where for any $1\leq i \leq n$, $B_i\in\cB_S$ and $A_i\in \cA^*_S$
so that $\cup^n_{i=1} A_i$ is a measurable partition of $[0,1]^{r([k])}$.
Let $\bI\in\sigma(S)^*$. Then
$$\phi^{-1}(M)\cap \bI=\cup^n_{i=1} (\phi^{-1}(A_i)\cap \bI)\cap
\phi^{-1}(B_i)\,.$$
Hence
$$\mu(\phi^{-1}(M)\cap \bI)=\sum^n_{i=1} \mu(\phi^{-1}(A_i)\cap \bI)
\mu(\phi^{-1}(B_i))\,.$$
Note that $\mu(\phi^{-1}(B_i))=Vol(M)$ and 
$\sum^n_{i=1} \mu(\phi^{-1}(A_i)\cap \bI)=\mu(\bI)$.
Therefore $\phi^{-1}(M)$ is independent from $\sigma(S)^*$.
By Lemma \ref{lemmafontos}, any set in $\cA_S$ which is independent
from $\cA^*_S$ can be approximated by sets in block-form, thus the
proof of the first part of our lemma follows. The second part 
can be proved completely similarly. \qed

\bigskip
\noindent
The following lemma is a baby-version of the Total Independence Lemma.
\begin{lemma} \label{baby}
For any $S\subseteq [k]$, let $X_S\in \cA_S$ such that $X_S$ is independent
form $\cA_S^*$. Then $\{X_S\}_{S\subseteq [k]}$ is a totally independent
system.
\end{lemma}
\proof
We need to prove that for any
set-system $\{S_i\}^r_{i=1}\subset r([k])$
\begin{equation}
\label{dada1}
Vol(\cap^r_{i=1} X_{S_i})=\prod^r_{i=1} Vol(X_{S_i})
\end{equation}
Let us proceed by induction. Suppose (\ref{dada1}) holds for a certain $r$.
Let $\{S_i\}^{r+1}_{i=1}\subset r([k])$ be a set-system and
suppose that $S_{r+1}$ is not a subset of $S_j$, for $1\leq j \leq r$.
It is enough to see that
\begin{equation}
\label{dada2}
Vol(X_{S_{r+1}}\cap\bigcap^r_{i=1} X_{S_i})=\prod^{r+1}_{i=1} Vol(X_{S_i})
\end{equation}
By Lemma \ref{lemmafontos} we may assume that $X_{S_{r+1}}$ is in
the block-form $\cup^n_{i=1}(A_i\cap B_i)$,
where $\cup^n_{i=1} A_i$ is a partition of $[0,1]^{r([k])}$ such that
$\{A_i\}^n_{i=1}$ are in the $\sigma$-algebra $\cC_S$ generated
by $\{\cA_S\}_{S\subset [k],\,S\neq S_{r+1}}$ and
$\{B_i\}^n_{i=1}\subset \cB_S$. Since $\cap^r_{i=1} X_{S_i}\in \cC_S$,
(\ref{dada2}) follows. 
\qed

\bigskip
\noindent
We shall need the auxilliary notion of
structure preserving measure algebra embeddings.
Let $\mathcal{L}^{r([k])}$ denote the measure algebra
associated to $([0,1]^{r([k])}, \mathcal{B},\lambda)$.
For any $S\subseteq [k]$ let $\mathcal{B}_S$ be the subalgebra
generated by the $S$-coordinate, that is for any $S\subset [k]$,
$\{B_T\}_{T\subseteq S}$ are jointly independent subalgebras
generating $\mathcal{A}_S$.
We say that an injective homomorphism
 $\Phi:\mathcal{L}^{r([k])}\mapsto \mathcal{L}^{r([k])}$ is 
a {\bf structure preserving embedding} if

\begin{enumerate}
\item $\Phi$ is measure preserving.
\item $\Phi(\mathcal{B}_S)\subset \mathcal{A}_S$ for any $S\subseteq [k]$.
\item $\Phi(\mathcal{B}_S)$ is independent of $\mathcal{A}^*_S$.
\item $\Phi\circ\pi=\pi\circ\Phi$ for every permutation in $S_k$.
\end{enumerate}
\begin{lemma} Let $\Phi:\mathcal{L}^{r([k])}\mapsto\mathcal{L}^{r([k])}$ be a
  (measure algebra) structure preserving embedding. Then $\Phi$ can be 
represented (see Lemma \ref{measurealgebra}) by a structure preserving map 
$\phi:[0,1]^{r([k])}\mapsto[0,1]^{r([k])}$. \label{lem42}
\end{lemma}

\begin{proof}
Let us consider the map $\Phi_{[i]}: \mathcal{B}_{[i]}\to \mathcal{A}_{[i]}$.
By the fourth axiom of structure preserving embeddings
the image of $\Phi_{[i]}$ consists of $S_{[i]}$-invariant elements.
We claim that we can represent $\Phi_{[i]}$ by maps
$\phi_{[i]}:[0,1]^{r([i])}\mapsto[0,1]$ such that $\phi^{-1}(I)$ is $S_{[i]}$
invariant for every 
measurable set $I\subseteq[0,1]$. First we represent $\Phi_{[i]}$ by a
 measurable map $\phi'_{[i]}$. Now Lemma \ref{inv1} implies that
 $S_{[i]}$ acts freely on $[0,1]^{r([i])}$ withe measurable sets
 $Q_1,Q_2,\dots,Q_{i!}$.
 Let $G=\cup_i Q_i$. If $x\in G$ then we define $\phi_{[i]}(x)$ 
as $\phi'{[i]}(\pi(x))$ where $\pi\in S_{[i]}$ is the unique permutation with
$\pi(x)\in Q_1$. If $x\in [0,1]^{r([i])}\setminus G$ the $\phi_{[i]}(x)$ is 
defined to be $0$. For a general set $S\in r([k])$ with $|S|=i$ 
we define $\phi_S(x)$ to be $\phi_{[i]}(\pi(x))$ where $\pi\in S_{[k]}$ is an
arbitrary permutation with $\pi(S)=[i]$. The $S_{[i]}$-invariance of 
$\phi_{[i]}$ guarantees that $\phi_S;[0,1]^{r(S)}\to [0,1]$ is well defined
and represents the map $\Phi_S:\mathcal{B}_{S}\to \mathcal{A}_{S}$.
It is easy to see that the map $\times_{S\in r([k])}\phi_S\circ L_S$ is a
 structure preserving map which represents $\Phi$. \qed
\end{proof}
\begin{lemma}\label{stepeq} Let $W\subseteq [0,1]^{r([k])}$ be an 
$l$-step hypergraphon and let $\phi:[0,1]^{r([k])}\mapsto[0,1]^{r([k])}$ be a
structure preserving map with $T=\phi^{-1}(W)$. 
Then there is a structure preserving equivalence (see Definition \ref{strequ})
$\psi$ such that $W\triangle\psi^{-1}(T)$ has measure $0$.
\end{lemma}

\begin{proof} Let $P_t^i$ denote the set 
$\phi_{[t]}^{-1}([(i-1)/l,i/l))\in\mathcal{A}_{[t]}$ for $t=1,2,\dots,k$. By
the definition of structure preserving maps the set $P_t^i$ is 
independent from $\mathcal{A}_{[t]}^*$, has measure $1/l$ and is 
symmetric under $S_{[t]}$.
Using Lemma \ref{incom2}, f
or every $t=1,2,\dots,k$ we construct a $\sigma$-algebra 
$\mathcal{C}_{[t]}\subseteq \mathcal{A}_{[t]}$ such that
\begin{enumerate}
\item $\mathcal{C}_{[t]}$ is an independent complement for 
$\mathcal{A}_{[t]}^*$ in $\mathcal{A}_{[t]}$
\item $P_t^i\in\mathcal{C}_{[t]}$ for $1\leq i\leq l$
\item Every set in $\mathcal{C}_{[t]}$ is invariant under the 
symmetric group $S_{[t]}$.
\end{enumerate}
In general, for a set $S\in r([k])$, we introduce $\mathcal{C}_S$ as
$\pi(\mathcal{C}_{[|S|]})$ where $\pi$ is an arbitrary 
permutation taking $\mathcal{A}{[|S|]}$ to $\mathcal{A}_S$. By the 
invariance of $\mathcal{C}_{[|S|]}$ this is well defined.

Now the system of $\sigma$-algebras $\{\mathcal{C}_S\}_{S\in r([k])}$ 
satisfies the following properties.
\begin{enumerate}
\item The $\sigma$-algebras $\mathcal{C}_S$ generate $[0,1]^{r([k])}$ 
where $S$ runs through the elements in $r([k])$.
\item $\cC_S\subset \cA_S$ and $\cC_S$ is independent form $\cA_S^*$.
That is by Lemma \ref{baby} the algebras $\cC_S$ are totally independent.
\end{enumerate}
Now let $\rho_{[t]}$ be a measure algebra isomorphism from $[0,1]$ to
$\mathcal{C}_{[t]}$ taking $[(i-1)/l,i/l)$ to $P_t^i$. Using the 
$S_{[k]}$ action we also define maps $\rho_S$ for every $S\in r([k])$ 
satisfying $\pi\circ\rho_S=\rho_S\circ\pi$ for every $\pi\in S_{[k]}$.
Since the algebras $\mathcal{C}_S$ are totally 
independent, by Lemma \ref{fremlin},
the product of the maps $\rho_S$ creates a measure algebra 
equivalence from $[0,1]^{r([k])}$ to itself which is a 
structure preserving equivalence.
\end{proof}

\begin{lemma}\label{stequiv} For every pair $U,W\subseteq [0,1]^{r([k])}$ 
of hypergraphons and $\epsilon>0$ there is a structure preserving 
equivalence $\phi:[0,1]^{r([k])}\mapsto[0,1]^{r([k])}$ 
such that $d_1(U,\phi^{-1}(W))\leq\delta_1(U,W)+\epsilon$.
\end{lemma}

\begin{proof} Let $T_1,T_2$ be a two $l$-step hypergraphons 
with $d_1(T_1,W)\leq\epsilon/8$ and $d_1(T_2,U)\leq\epsilon/8$. We know that
there are two structure preserving maps $\psi_1$ and $\psi_2$ 
such that $d_1(\psi_2^{-1}(U),\psi_1^{-1}(W))\leq\delta_1(U,W)+\epsilon/8$. By
Lemma \ref{stepeq} there are structure preserving equivalences 
$\rho_1$ and $\rho_2$ \\ with $d_1(\rho_1^{-1}(T_1),\psi_1^{-1}(T_1))=0$ 
and $d_1(\rho_2^{-1}(T_2),\psi_2^{-1}(T_2))=0$.
Now $$d_1(\rho_1^{-1}(T_1),\rho_2^{-1}(T_2))\leq\delta_1(U,W)+\epsilon/4.$$
By Lemma \ref{functor},
 $\rho_2\circ\rho_1^{-1}$ is a structure preserving equivalence 
that takes $W$ into a set whose distance from $U$ is at most 
$\delta_1(U,W)+\epsilon$.
\end{proof}
\subsection{A concentration result for $W$-random graphs}

\begin{theorem}[Concentration]\label{conc} Let 
$W\subseteq[0,1]^{r([k])}$ be a hypergraphon. 
Then
$${\rm Pr}(|t_0(F,\mathbb{G}(W,[n]))-t(F,W)|\geq\epsilon)\leq 
2\exp\bigl(-\frac{\epsilon^2n}{2|V(F)|^2}\bigr).$$
\end{theorem}

The proof of the lemma is identical with the proof of Theorem 2.5
in \cite{LSZ},
that was used for the case $k=2$.
For the sake of completeness we repeat the proof.

\begin{proof}
Let us consider the system of random hypergraph models 
$G_1,G_2,\dots,G_n$ such that the distribution of $G_n$ is 
$\mathbb{G}(W,[n])$ and $G_i$ is the sub hypergraph in $G_n$ 
induced by $[i]$. It is clear that the distribution on $G_i$ is
 the same as $\mathbb{G}(W,[i])$.
Let $F$ be a fixed $k$-uniform hypergraph on the vertex set $[r]$.
For any injective map $\psi:[r]\mapsto[n]$ we denote by $A_\psi$ the 
event that $\psi$ is a homomorphism from $F$ to $G_n$. Let
$$B_m=\frac{(n-r)!}{n!}\sum_\psi{\rm Pr}(A_\psi~|~G_m).$$
The sequence $B_0,B_1,\dots,B_n$ is a martingale,
where $B_0={\rm Pr}(A_\psi)=t(F,W)$ and \\ $B_n=
{\rm Pr}(A_\psi~|~G_n)$ is $1$ if $\psi$ is a homomorphism and $0$
elsewhere. 
This implies that $B_n=t_0(F,G_n)$. Now we have that
$$|B_m-B_{m-1}|\leq\frac{(n-r)!}{n!}\sum_{\psi}\bigl|{\rm Pr}
(A_\psi~|~G_m)-{\rm Pr}(A_\psi~|~G_{m-1})\bigr|.$$
The terms in the sum for which $m$ is in not in the range of 
$\psi$ are $0$ and all the other terms are at most one. The 
number of terms of the second type is $r\frac{(n-r)!}{(n-1)!}$ and so 
$|B_m-B_{m-1}|\leq r/n$. By applying Azuma`s inequality we get that
$${\rm Pr}(|t_0(F,G_n)-t(F,W)|\geq\epsilon)={\rm
  Pr}(|B_n-B_0|\geq\epsilon)\leq
2\exp\bigl(\frac{-\epsilon^2}{2n(r/n)^2}\bigr)=
2\exp\bigl(-\frac{\epsilon^2n}{2r^2}\bigr)$$
\qed
\end{proof}

\begin{theorem}[Convergence]\label{asconv} The sequences 
$t_0(F,\mathbb{G}(W,[n]))$ and $t(F,\mathbb{G}(W,[n])$ 
converge to \\ $t(F,W)$ with probability one as $n$ goes to infinity.
\end{theorem}

\begin{proof} The convergence of $t_0(F,\mathbb{G}(W,[n]))$ follows from Lemma
  \ref{conc} and the Borel-Cantelli lemma since for every
 fixed $\epsilon>0$ the sum of the right hand side in the inequality is finite.
\qed
\end{proof}

\subsection{Proof of the Uniqueness Theorems}

Let $\xo$ be the ultra product of the sets $[n]$. Let $Z_n$ be the random
variable which is a random point in $[0,1]^{r([n],k)}$ 
with uniform distribution as in Section \ref{sub26} and let $\{\tau^n\}:
[n]_0^k\to [0,1]^{r([k])}$ be the associated random coordinate systems.
The ultraproduct function $\tau=[\{\tau^n\}_{n=1}^\infty]$ on $\xo^k$ 
will also be called random coordinate system

\begin{lemma}\label{randcorsep} The random coordinate system 
$\tau:\xo^k\mapsto [0,1]^{r([k])}$ is a separable realization with
 probability one.
\end{lemma}

\begin{proof}
Let $I\subseteq [0,1]$ be a measurable set and $S\in r([k])$. It is clear from
the definition that $\tau_S^{-1}(I)$ is in $\sigma(S)$. We show that 
(with probability one) $\tau_S^{-1}(I)$ is independent from $\sigma(S)^*$ 
and has measure $\lambda(I)$.

Let $I_{a,b}$ be an open interval with rational endpoints $a,b$. 
Let $\hat{\bf I}_{a,b}$ denote the ultra product
$[\{(\tau^n_S)^{-1}(I_{a,b})\}_{n=1}^\infty]$. By Proposition \ref{random}
(and the remark after the proposition)
we have that almost 
surely $\hat{\bf I}_{a,b}$ has measure $b-a$ and is independent 
from $\sigma(S)^*$.
Then we have 
$$\hat{\bf I}_{a+\epsilon,b-\epsilon}\subseteq\tau_S^{-1}(I_{a,b})
\subseteq\hat{\bf I}_{a,b}$$
for every small enough
 rational number $\epsilon>0$. Since there are only countable many
rational numbers this holds simultaneously for every rational 
number with probability $1$.
This implies that $\tau_S^{-1}(I_{a,b})$ has measure $b-a$ and is 
independent from $\sigma(S)^*$ with probability $1$. Since $\tau_S$ 
is measurable and measure preserving on rational intervals it has to be
measure preserving on Lebesgue sets. By approximating an arbitrary measurable
 sets by unions of disjoint intervals we get the independence 
from $\sigma(S)^*$.

Now let $B\subseteq [0,1]^{r([k])}$ be a box of the form 
$\prod_{S\in r([k])} I_S$ where $I_S$ is an interval with rational endpoints.
 The measure of $B$ is equal to $\prod_{S\in r([k])}\lambda(I_S)$. 
The set $\tau^{-1}(B)$ is equal to
$$\cap_{S\in r([k])}\tau_S^{-1}(I_S)\,.$$ Therefore using the total 
independence theorem we obtain that with probability one
$\tau^{-1}(B)=\lambda(B)$. Again this holds simultaneously for every rational
interval 
system with probability $1$. As a consequence $\tau$ is 
almost surely a measure preserving map.

The symmetry on $\tau$ under $S_k$ is clear from its definition.\qed
\end{proof}

\medskip

\begin{lemma}\label{randsep} Let $W$ be a hypergraphon. 
Then with probability one the ultraproduct \\
$\bH=[\{\mathbb{G}(W,[n])\}_{n=1}^\infty]\subseteq \xo^k$ has a separable
realization $\phi:\xo^k\mapsto [0,1]^{r([k])}$ such that 
$\bH\triangle\phi^{-1}(W)$ has measure $0$.
\end{lemma}

\begin{proof} We will use that the set 
$\bH=[\{\mathbb{G}(W,[n])\}_{n=1}^\infty]$ can be written as the ultraproduct
$[\{(\tau^n)^{-1}(W)\}_{n=1}^\infty]$. Our goal is to prove that 
almost surely $\bH\triangle\tau^{-1}(W)$ has measure $0$.
First by applying Theorem \ref{asconv} to a single hyperedge $F$
 we deduce that $\bH$ has measure $\lambda(W)$ with probability one.

If $W$ is open then $\tau^{-1}(W)$ is contained in $\bH$ and, by Lemma
\ref{randcorsep}, has measure $|W|$ with probability $1$. This means that 
with probability $1$ the set $\tau^{-1}(W)\triangle \bH$ has measure $0$.

For an arbitrary measurable set $W\subseteq [0,1]^{r([k])}$ and $\epsilon>0$
there are open sets $O_1$ and $O_2$ in $[0,1]^{r([k])}$ such that 
$O_1\setminus O_2\subseteq W\subseteq O_1$ and $|O_2|\leq\epsilon$. 
We have that
$$(\tau^n)^{-1}(O_1)\setminus(\tau^n)^{-1}(O_2)\subseteq 
(\tau^n)^{-1}(W)\subseteq(\tau^n)^{-1}(O_1)$$
and thus by taking the ultra product
$$[\{(\tau^n)^{-1}(O_1)\}_{n=1}^\infty]
\setminus[\{(\tau^n)^{-1}(O_2)\}_{n=1}^\infty]
\subseteq\bH\subseteq[\{(\tau^n)^{-1}(O_1)\}_{n=1}^\infty\}].$$
Using our observation about open sets and that 
$\tau$ is measure preserving with probability $1$ 
we obtain that the measure of $\bH\triangle\tau^{-1}(W)$ is at 
most $\epsilon$. By Lemma \ref{randcorsep} the proof is complete.\qed
\end{proof}

\bigskip

\noindent{\bf Proof of Theorem \ref{uni1}}.

\noindent
Let $U$ and $W$ be two functions with $\delta_w(U,W)=0$. This means that
$\mathbb{G}(U,[n])$ and $\mathbb{G}(W,[n])$ are equal to the same distribution
$Z_n$. 
Let $\bH=[\{Z_n\}_{n=1}^\infty]$. By Lemma \ref{randsep} with probability one
there are two separable realizations $\phi_1,
\phi_2:\xo^n\mapsto [0,1]^{r([k])}$ such that $\phi_1^{-1}(U)$, 
$\phi_2^{-1}(W)$ and $\bH$ differ only in a zero measure set. Let
$\mathcal{A}$ denote the separable sigma algebra generated by $\phi_1$ 
and $\phi_2$. By the Euclidean correspondence (Theorem \ref{eucc}) there is a
separable realization $\phi_3:\xo^k\mapsto [0,1]^{r([k])}$ 
corresponding to the algebra $\mathcal{A}$. The maps $\phi_1$ and $\phi_2$ 
define unique structure preserving maps $\psi_1$ and $\psi_2$ on the 
measure algebra $\mathcal{L}^{r([k])}$ such that $(\phi_3)^{-1}\psi_i(S)$ is
equivalent with $\phi_i^{-1}(S)$. This means that $\psi_1(U)=\psi_2(W)$ 
in the measure algebra $\mathcal{L}^{r([k])}$.
Therefore by Lemma \ref{lem42} our theorem follows. \qed

\bigskip

\noindent{\bf Proof of Theorem \ref{uni2}.}

\noindent
By the previous theorem, if $\delta_w(U,V)=0$, then
$\delta_1(U,V)=0$. On the other hand, if $\delta_1(U,V)=0$ then
by the fact that $\delta_w(U,V)\leq d_1(U,V)$ and Lemma \ref{lem41},
$\delta_w(U,V)=0$. \qed

\subsection{The Counting Lemma}

 Let $\mathcal{C}\subseteq [l]^{r([k])}$ be a symmetric 
combinatorial structure. Let $V$ be a finite set. 
An $(l,k)$-map is a function from $r(V,k)$ to $[l]$. If $E\subseteq V$ has
$k$-elements then the restriction of an $(l,k)$-map $f$ to $E$ is an 
element $x$ in $[l]^{r(E)}$. By specifying an arbitrary bijection $g$ 
between $E$ and $[k]$ we can also represent $x$ by an element $x'$ in
$[l]^{r([k])}$. The $S_k$-orbit of $x'$ does not depend on $g$ and so we can 
talk about the $S_k$-orbit determined by the restriction of $f$ to $E$.

 Let $F$ be a $k$-uniform hypergraph on $V$, 
let $\mathcal{C}\subseteq[l]^{r([k])}$ be a symmetric combinatorial structure
and let $f$ be an $(l,k)$-map on $V$. We say that $f$ is 
a {\bf homomorphism} form $F$ to $\mathcal{C}$ if the 
restriction of $f$ to any edge of $F$ determines an $S_k$ 
orbit which is in $\mathcal{C}$.

 The {\bf homomorphism density} $t(F,\mathcal{C})$ is the probability that a
 random $(l,k)$-map on $V$ is a homomorphism. Note that here 
we take the uniform probability distribution on all $(l,k)$-maps. 
For technical reasons we will also need the number $t(F,\mathcal{C},P)$ which
is the probability that an $(l,k)$-map chosen with distribution $P$ 
on $[l]^{r(V,k)}$ (the set of all $(l,k)$-functions) is a homomorphism.

Let $\mathcal{H}$ be an $l$-hyperpartition on a finite set $U$. 
Every injective map $g:V\mapsto U$ induces an $(l,k)$-map $f_g$ on $V$ such
that for a set $S\in r(V,k)$ the value $f(S)$ is the index $i$ of the 
partition set $P_{|S|}^i$ containing $g(S)$. Let $D(V,\mathcal{H})$ denote the
probability distribution of $f_g$ if $g$ is chosen uniformly at 
random from all the injective maps $g:V\mapsto U$. Using this 
notation the following lemma follows immediately from the definitions.

\begin{lemma}\label{rmwd} Let $\mathcal{C}\subseteq[l]^{r([k])}$ be a
  symmetric combinatorial structure and let $H$ be a hypergraph on the set $U$ 
which is the union of $\mathcal{H}$-cells with coordinates in
$\mathcal{C}$. Then the probability $t^0(F,H)$ that a random injective map
$g:V\mapsto U$ is a 
homomorphism from $F$ to $H$ is equal to $t(F,\mathcal{C},D(V,\mathcal{H}))$.
\end{lemma}

\begin{theorem}[Counting Lemma]\label{counting} Let $\{U_i\}_{i=1}^\infty$ be
  increasing finite sets with $l$-hyperpartitions \\
$\{\mathcal{H}_i\}_{i=1}^\infty$ such that $\mathcal{H}_i$ is 
$\epsilon_i$-regular and $\delta_i$-equitable with 
$\lim_{i\to\infty}\epsilon_i=\lim_{i\to\infty}\delta_i=0$. 
Let furthermore  $\mathcal{C}\subseteq[l]^{r([k])}$ be a 
symmetric combinatorial structure and $H_i$ be the union 
of $\mathcal{H}_i$-cells with coordinates in $\mathcal{C}$.
Then for every finite hypergraph $F$ we have that
$$\lim_{i\to\infty}t(F,H_i)=t(F,\mathcal{C}).$$
\end{theorem}

\begin{proof}
Let $V$ denote the vertex set of $F$. 
Since $\{U_i\}_{i=1}^\infty$ is an increasing sequence of sets
we have that
$$\lim_{i\to\infty}t^0(F,H_i)=\lim_{i\to\infty}t(F,H_i).$$
Now by Lemma \ref{rmwd} it suffices to show that
$\lim_{i\to\infty}D(V,\mathcal{H}_i)$ is the uniform 
distribution on $[l]^{r(V,k)}$.
We proceed by contradiction. By choosing an appropriate 
subsequence of $\{\mathcal{U}_i\}_{i=1}^{\infty}$ we 
can assume that the limit of $D(V,\mathcal{H}_i)$ exists and 
it is not uniform. This means that there is a 
function $f:r(V,k)\mapsto [l]$ such that
$$\lim_{i\to\infty}p_i\neq l^{-|r(V,k)|}\,,\quad\mbox{where}\,\,
p_i=P(f_g=f~|~g:V\mapsto U_i,~g~{\rm is~injective})$$
holds.
The set of all injective maps from $V$ to $U_i$ can be
 represented as the collection of elements in $U_i^V$ with 
no repetitions in the coordinates. This subset in $U_i^V$ has
 relative density tending to $1$ as $i$ goes to infinity.
Now let $T_i\subseteq U_i^V$ defined by
$$T_i:=\bigcap_{S\in r(V,k)}\pi^{-1}_S(P_{|S|}^{f(S)})\,.$$
For $S\subset r(V,k)$, $\pi_S:U_i^V\to U_i^{[|S|]}$ is defined as
$L_{\rho_S}\circ L_S$, where $L_S:U_i^V\to U_i^S$ is the natural 
projection and $L_{\rho_S}$ is given by a bijection $\rho_S:S\to[|S|]$.
Here $P_{|S|}^{f(S)}$ denotes the corresponding
partition set in $\mathcal{H}_i$. Since $P_{|S|}^{f(S)}$ is symmetric in its
coordinates
 the set $T_i$ is independent of the concrete choice of the 
bijections $\rho_S$.
Let $\xo$ denote the ultraproduct $[\{U_i\}_{i=1}^\infty]$ and let $\bH^f$ 
denote the ultraproduct $[\{T_i\}_{i=1}^\infty]\subseteq\xo^V$.
Furthermore for every $S\in r(V,k)$ let $\bH^f_S$ denote the ultralimit
 of the partition sets $\pi^{-1}_S(P_{|S|}^{f(S)})$ from $\mathcal{H}_i$
 where $i$ tends to infinity.
Then
$$\bH^f=\bigcap_{S\in r(V,k)}\bH^f_S\,.$$
Also, the measure of $\bH^f$ is equal to 
$\lim_\omega p_i=\lim_{i\to\infty}p_i$.
Now the condition $\lim_{i\to\infty}\epsilon_i=\lim_{i\to\infty}\delta_i=0$
implies that for every $S\in r(V,k)$ the set $\bH^f_S$ has 
measure $l^{-1}$ and that $\bH^f_S\in\sigma(S)^*$.
The total independence theorem implies that the
 measure of $\bH$ is $l^{-|r(V,k)|}$ providing a contradiction.\qed
\end{proof}

\subsection{Equivalence of convergence notions and the Inverse Counting Lemma}

Let $W\subseteq [0,1]^{r([k])}$ be a hypergraphon. 
We say that a sequence of hypergraphs $\{H_i\}_{i=1}^{\infty}$ is {\bf
  structurally converges} to $W$ if for every $l$-step hypergraphon $U$ with 
$\delta_1(W,U)\leq\epsilon$ and combinatorial structure $\mathcal{C}$ 
there is a sequence of $l$-hyperpartitions $\mathcal{H}_i$ on the vertex 
sets of $H_i$ such that
\begin{enumerate}
\item $\mathcal{H}_i$ is $\delta_i$-regular and $\delta_i$-equitable with 
$\lim_{i\to\infty}\delta_i=0$
\item The union $T_i$ of $\mathcal{H}_i$-cells with coordinates 
in $\mathcal{C}$ satisfies $\limsup_{i\to\infty} d_1(T_i,H_i)\leq\epsilon$
\end{enumerate}

\begin{definition} We say that an $l$-step hypergraphon $U$ 
with combinatorial structure $\mathcal{C}$ is $(\epsilon,\delta)$-close 
to a hypergraph $H$ if there is an $l$-hyperpartition 
$\mathcal{H}$ on the vertex set of $H$ such that
\begin{enumerate}
\item $\mathcal{H}$ is both $\delta$-regular and $\delta$-equitable.
\item The union $T_U$ of $\mathcal{H}$-cells with combinatorial
 structure $\mathcal{C}$ satisfies $d_1(H,T_U)\leq \epsilon$.
\end{enumerate}
\end{definition}

\begin{theorem}\label{three} For an increasing sequence $\{H_i\}_{i=1}^\infty$ of 
$k$-uniform hypergraphs the following statements are equivalent:
\begin{enumerate}
\item $\{H_i\}$ is strongly convergent
\item $\{H_i\}$ is weakly convergent
\item $\{H_i\}$ structurally converges to a hypergraphon $W$ 
which is also the weak limit of $\{H_i\}$.
\end{enumerate}
\end{theorem}

\begin{proof} Let us start with (2) implies (3).
By Theorem \ref{lobtetel} we know that there is a hypergraphon $W$ such 
that $lim_{i\to\infty} t(F,H_i)=t(F,W)$.
Assume by contradiction that $\{H_i\}$ is not structurally
 convergent to $W$. Then for some $\epsilon>0$ there is a $\delta>0$, an
 $l$-step 
hypergraphon $U$ of combinatorial structure $\mathcal{C}$ with
 $\delta_1(U,W)\leq\epsilon$ and an infinite subsequence $\{J_i\}$ of
 $\{H_i\}$ 
such that none of the elements of $\{J_i\}$ is 
$(\delta,\epsilon+\delta)$-close to $T$. Let ${\bf J}$ be the ultra product
hypergraph $[\{J_i\}_{i=1}^\infty]\subseteq\xo^k$ and 
let $\phi:\xo^k\mapsto[0,1]^{r([k])}$ be a separable
 realization of ${\bf J}$. That is for some $V\subset [0,1]^{r([k])}$,
${\bf J}\triangle \phi^{-1}(V)$ has measure zero. By Lemma \ref{lemma31},
$\delta_w(V,W)=0$.
By Theorem \ref{uni2}, $\delta_1(V,W)=0$ and consequently 
$\delta_1(U,V)\leq\epsilon$. By Lemma \ref{stepeq} there exists a measure preserving
equivalence $\rho$ with $d_1(\rho^{-1}(U),V)\leq\epsilon+\delta/2$. 
This means that $(\rho\circ\phi)^{-1}(U)\triangle{\bf J}$ has measure 
at most $\epsilon+\delta/2$.
By Lemma \ref{functor} $\rho\circ\phi$ is a separable realization, hence 
$(\rho\circ\phi)^{-1}(U)$ is a cell system with combinatorial 
structure $\mathcal{C}$ of a $0$-regular and $0$-equitable 
hyperpartition on $\xo$. This leads to a contradiction.

The implication $(3)\Rightarrow(1)$ is trivial.

The implication $(1)\Rightarrow(2)$ follows 
from the Counting Lemma (Theorem \ref{counting}). 
Let us fix a $k$-uniform hypergraph $F$ on the 
vertex set $V$ and with edge set $E$. According to 
the definition of strong convergence for every $\epsilon>0$ there is a fixed
combinatorial structure $\mathcal{C}$ and modifications 
$H_i'$ of $H_i$ with an at most $\epsilon$-density 
edge set such that every $H_i'$ is the union of the 
cells with coordinates in $\mathcal{C}$ of some hyperpartition which is
getting more and more regular and balanced as $i$ tends to infinity.
The Counting Lemma implies that 
$\lim_{t\to\infty}t(F,H_i')=t(F,\mathcal{C})$. 
On the other hand $|t(F,H_i)-t(F,H_i')|\leq |E|\epsilon$.
 Using this inequality for every $\epsilon>0$, we obtain the convergence of
 $t(F,H_i)$. \qed
\end{proof}

 The following immeadiate corollary states that if two hypergraphs 
have similar 
sub-hypergraph densities then they have similar regular partitions. 
\begin{corollary}[Inverse Counting Lemma]\label{inverse}
Fix $k>0$. Then for any $\epsilon>0$ there exist 
positive constants
$\delta=\delta(\epsilon), C=C(\epsilon), 
N=N(\epsilon)$ such that
if $H_1$, $H_2$ are two $k$-uniform hypergraphs, $|V(H_1)|\geq N, |V(H_2)|\geq
N$ and $\delta_1(H_1,H_2)<\delta$,
 then there exists an $l$-hyperpartition $U$, $1<l\leq C$ so that
both hypergraphs are $(\epsilon,\epsilon)$-close to $U$.
\end{corollary}
We also have a corollary of the Counting Lemma, using the
notion of $(\epsilon, \delta)$-closeness.
\begin{corollary}[Counting Lemma Finitary Version]\label{counting2}
For any finite $k$-uniform hypergraph $F$,$l$-step hypergraphon
$U$ and $\epsilon>0$ there is a constant $\delta=\delta(F,U,\epsilon)$ 
such that if
a $k$-uniform hypergraph $H$ is $(\delta,\delta)$-close to $U$ then
$$|t(F,U)-t(F,H)|<\epsilon\,.$$ (see also \cite{NRS})
\end{corollary}

\section{The proof of the Total Independence Theorem}
\label{prooftotal}
Let $\{X_i\}_{i=1}^\infty$ be finite sets as in Section \ref{prelim} and
$f_i:X_i\to[-d,d]$ be real functions, where $d>0$. Then one can define
a function $\fb:\xo\to[-d,d]$ whose value at $\overline{p}=
[\{p_i\}^\infty_{i=1}]$
is the ultralimit of $\{f_i(p_i)\}^\infty_{i=1}$. We say that
$\fb$ is the ultraproduct of the functions $\{f_i\}^\infty_{i=1}$.
We shall use the notation $\fb=[\{f_i\}^\infty_{i=1}]$. Note that
the characteristic function of the ultraproduct of sets is exactly
the ultraproduct of their characteristic functions.
 From now on
we call such bounded functions {\bf ultraproduct functions}.
\begin{lemma}\label{ultralimitfunction}
The ultraproduct functions are  measurable on $\xo$ and
$$\int_{\xo} \fb \,d\muo=\limo \frac{\sum_{p\in X_i} f_i(p)}{|X_i|}\,.$$
\end{lemma}
\begin{proof} Let $-d\leq a\leq b\leq d$ be real numbers. It is enough to
prove that
$\fb_{[a,b]}=\{\overline{p}\in \xo\mid\,
a\leq \fb(\overline{p})\leq b\}$ is measurable.
Let $f_{[a,b]}^i=\{p\in X_i\mid a\leq f_i(p)\leq b\}\,.$
Note that $[\{f_{[a,b]}^i\}^\infty_{i=1}]$ is
not necessarily equal to $\fb_{[a,b]}$. Nevertheless if
$$P_n:= [\{f^i_{[a-\frac{1}{n},b+\frac{1}{n}]}\}^\infty_{i=1}]\,,$$
then $P_n\in\cP$ and $\fb_{[a,b]}=\cap^\infty_{n=1} P_n$. This shows that
$\fb_{[a,b]}$ is a measurable set. Hence the function $\fb$ is measurable.

\noindent
Now we prove the integral formula.
Let us consider the function $g_i$ on $X_i$ which takes the value
$\frac{j}{2^k}$ if $f_i$ takes a value not smaller than
$\frac{j}{2^k}$ but less than $\frac{j+1}{2^k}$ for
$-N_k\leq j \leq N_k$, where $N_k=d 2^k+1$.
Clearly
$ |[\{g_i\}^\infty_{i=1}] - \fb|\leq \frac{1}{2^k}$ on $\xo$. Observe that
$\gb=[\{g_i\}^\infty_{i=1}]$ is a measurable step-function on $\xo$ taking the
value $\frac{j}{2^k}$ on $C_j= [\{f^i_{[\frac{j}{2^k},
\frac{j+1}{2^k})}\}^\infty
_{i=1}]$. Hence,
$$\int_X \gb\,d\muo= \sum^{N_k}_{-N_k}\frac{j}{2^k}\mu(C_j)=
\limo\left(\sum^{N_k}_{j=-N_k} \frac{j} {2^k}
\frac{|f^i_{[\frac{j}{2^k},\frac{j+1}{2^k})}|}{|X_i|}\right)\,.$$
Also, $|\gb-\fb|\leq\frac{1}{2^k}$ on $\xo$ uniformly, that is
$|\int_{\xo} \fb\,d\muo - \int_{\xo} g\,d\muo|\leq \frac{1}{2^k}\,.$
Notice that for any $i\geq 1$
$$\left|\sum^{N_k}_{j=-N_k}
\frac{|f^i_{[\frac{j}{2^k},\frac{j+1}{2^k})}|}{|X_i|}\frac{j} {2^k}-
\frac{\sum_{p\in X_i} f_i(p)}{|X_i|}\right|\leq \frac{1}{2^k}\,.$$
Therefore for each $k\geq 1$,
$$\left|\int_\xo \fb\,d\muo-\limo 
\frac{\sum_{p\in X_i} f_i(p)}{|X_i|}\right|\leq
\frac{1}{2^{k-1}}\,.$$
Thus our lemma follows.
\qed
\end{proof} \vskip 0.2in

\begin{proposition}\label{tetel2}
For every measurable function $\fb:\xo\to[-d,d]$, there exists a sequence
of functions $f_i:X_i\to [-d,d]$ such that
the ultraproduct of the sequence $\{f_i\}_{i=1}^\infty$ is
almost everywhere equal to $\fb$. That is any element of
$L^\infty(\xo,\bo,\muo)$ can be represented by an ultraproduct function.
\end{proposition}
\begin{proof}
Recall a standard result of measure theory. If $\fb$ is a bounded measurable
function on $\xo$, then there exists a sequence of bounded
stepfunctions $\{h_k\}^\infty_{k=1}$ such that
\begin{itemize}
\item
$\fb=\sum^\infty_{k=1} h_k$
\item  $|h_k|\leq \frac{1}{2^{k-1}}$, if $k>1$.
\item $h_k=\sum^{n_k}_{n=1} c^k_n \chi_{A^k_n}$, where
$\cup^{n_k}_{n=1}A^k_n=\xo$ is a
measurable partition, $c^k_n\in\bR$ if $1\leq n \leq n_k$.
\end{itemize}
Now let $B^k_n\in\cP$ such that $\muo(A^k_n\triangle B^k_n)=0$.
We can suppose that $\cup^{n_k}_{n=1} B^k_n$ is a partition of $\xo$.
Let $h'_k=\sum^{n_k}_{n=1} c^k_n \chi_{B^k_n}$ and
$\fb'=\sum^\infty_{k=1} h'_k$.
Then clearly $\fb'=\fb$ almost everywhere. 
We show that $\fb'$ is an ultraproduct
function.

\noindent
Let $B^k_n=[\{B^k_{n,i}\}^\infty_{i=1}]$.
We set $T_k\subset \bN$ as the set of integers $i$ for which
$\cup_{n=1}^{n_k} B^k_{n,i}$ is a partition of $X_i$. Then obviously,
$T_k\in\omega$.
Now we use our diagonalizing trick again. If $i\notin T_1$ let $s_i\equiv 0$.
If $i\in T_1, i\in T_2,\dots,i\in T_k, i\notin T_{k+1}$ then
define $s_i:=\sum^k_{j=1}(\sum^{n_j}_{n=1} c^j_n \chi_{B^j_{n,i}})\,.$
If $i\in T_k$ for each $k\geq 1$ then set
$s_i:=\sum^i_{j=1}(\sum^{n_j}_{n=1} c^i_n \chi_{B^j_{n,j}})\,.$
Now let $\overline{p}\in B^1_{j_1}\cap B^2_{j_2}\cap\dots\cap B^k_{j_k}$.
Then
$$|(\limo s_i)(\overline{p})-\fb'(\overline{p})|\leq\frac{1}{2^{k-1}}\,.$$
Since this inequality holds for each $k\geq 1$, $\fb'\equiv
[\{s_i\}^\infty_{i=1}]$. \qed
\end{proof} \vskip 0.2in
\begin{lemma}
\label{l14}
\label{vetites} Let $A,B\subseteq [k]$ and
let $\fb:\xok\to\mathbb{R}$ be a $\sigma(B)$-measurable ultraproduct
 function.
 Then for all $y\in \xo^{\ac}$ the function
$\fb_y$ is $\sigma(A\cap B)$-measurable, where $\ac$ denotes
the complement of $A$ in $[k]$ and $\fb_y(x)=\fb(x,y)$.
\end{lemma}
\begin{proof}
Let $\fb:\xok\to\bR$ be a $\sigma(B)$-measurable ultraproduct
function. Note that there exist functions $f_i:X_{i,[k]}\to\bR$
depending only on the $B$-coordinates such that $\fb$ is the ultraproduct
of $\{f_i\}_{i=1}^\infty$. Indeed, let $\fb$ be the ultraproduct of
the functions $g_i$. For  $x\in X_{i,B}$, let
$f_i(x,t):=\frac {\sum_{z\in X_{i,B^c}} g_i(x,z)}{|X_{i,B^c}|}$.
Then $f_i$ depends only on the $B$-coordinates. Also by the integral
formula of Lemma \ref{ultralimitfunction} , $\limo f_i=\fb$.
Let $y\in\xo^{\ac}$,
$y=[\{y_i\}^\infty_{i=1}]$. Then $\fb_y$ is the ultraproduct of
the functions $f^{y_i}_i$. Clearly $f^{y_i}_i$ depends only on the $A\cap
B$-coordinates, thus the ultraproduct $\fb_y$ is $\sigma(A\cap B)$-measurable.
\qed \end{proof} \vskip 0.2in
\begin{proposition}[Fubini's Theorem]  Let $A\subseteq [k]$ and
let $\fb:\xo^{[k]}\to\mathbb{R}$ be a 
bounded $\sigma([k])$-measurable function.
 Then for almost all $y\in \xo^{A^c}$, $\fb_y(x)$ is a measurable function
on $\xo^A$ and the function \\ $y\to \int_{\xo^A}\fb_y(x) d\mu_A(x)$
is $\xo^{A^c}$-measurable. Moreover:
$$\int_{\xok}\fb(p) d\mu_{[k]}(p)=\int_{\xo^{\ac}}
\left(\int_{\xo^A}\fb_y(x) d\mu_A(x)
\right) d\mu_{\ac} (y)\,.$$
\end{proposition}
\proof First
let $\fb$ be the ultraproduct of $\{f_i:X_{i,[k]}\to\mathbb{R}\}^\infty_{i=1}$.
Define the functions $\overline{f_i}:X_{i,\ac}\to [-d,d]$ by
$$\overline{f_i}(y)=|X_{i,A}|^{-1}\sum_{x\in X_{i,A}}f_i(x,y).$$
By Lemma \ref{ultralimitfunction}
$$\limo \overline{f_i}(y)=\int_{\xo^A} \fb(x,y) \,d\mu_A(x)\,.$$
Applying Lemma \ref{ultralimitfunction}
again for the functions $\overline{f_i}$, we obtain that
$$\limo |X_{i,\ac}|^{-1}\sum_{y\in X_{i,\ac}}
\overline{f_i}(y)
=\int_{\xo^{\ac}}\left(\int_{\xo^A}\fb(x,y) d\mu_A(x)\right) d\mu_{\ac}(y)\,.$$
Then our proposition follows, since
$$|X_{i,\ac}|^{-1}\sum_{y\in X_{i,\ac}}
\overline{f_i}(y)=\frac{\sum_{p\in X_i} f_i(p)}{|X_i|}\,.$$
Now let $\fb$ be an arbitrary bounded  $\sigma([k])$-measurable function.
Since there exists an ultraproduct function $\gb$ that is a zero measure
perturbation of $\fb$ it is enough to prove the following lemma:
\vskip0.1in
\noindent
{\bf Lemma:} {\it Let $Y\subset \xo^{[k]}$ be a measurable set of zero
measure, then for almost all $y\in \xo^{A^c}$,
$$\{x\in \xo^A\,\mid\, \xo^{A}\times y \in Y\}$$ has measure zero.}
\vskip0.1in
\noindent
{\bf Proof:}
Since $Y$ is a set of zero measure, there exists sets $Z_n\in \cP_{[k]}$
such that
\begin{itemize}
\item
$\mu_{[k]}(Z_n)\leq \frac{1}{4^n}$
\item
$Y\subset Z_n$.
\end{itemize}
Let ${ L}_n\subset \xo^{A^c}$ be the set of points $y$ in $\xo^{A^c}$
such that $$ \mu_{A}(\{x\in \xo^A\,\mid\, \xo^{A}\times y \in Z_n\})\geq
\frac{1}{2^n}\,.$$
Since Fubini's Theorem holds for ultraproduct functions it is easy to see
that $\mu_{A^c}({L}_n)\leq \frac{1}{2^n}$.
Thus by the Borel-Cantelli Lemma almost all $y\in\xo^{A^c}$ is contained
only in finitely many sets $L_n$. Clearly,
for those $y$, $\{x\in \xo^A\,\mid\, \xo^{A}\times y \in Y\}$ has
measure zero . \qed
\begin{proposition}[Integration Rule]
Let $g_i:\xok\to\mathbb{R}$ be bounded
$\sigma(A_i)$-measurable functions for $i=1,2,\dots,m$. Let $B$ denote the
$\sigma$-algebra generated by
$\sigma(A_1\cap A_2),\sigma(A_1\cap A_3),\dots,\sigma(A_1\cap A_m)$. Then
$$\int_{\xok} g_1g_2\dots g_m\,d\mu_{[k]}=
\int_{\xok} E(g_1|B)g_2g_3\dots g_m\,d\mu_{[k]}\,.$$
\end{proposition}

\proof
First of all note that $E(g_1\mid B)$ does not depend on the
$A_1^c$-coordinates.
 By Fubini's Theorem,
$$\int_{\xok} g_1g_2g_3\dots
g_m\,d\mu_{[k]}=\int_{\xo^{A_1^c}}\left(\int_{\xo^{A_1}}
g_1(x) g_2(x,y)\dots g_m(x,y)\,d\mu_{A_1}(x)\right)d\mu_{A_1^c}(y)\, .$$
Now we obtain by Lemma \ref{vetites} that for all $y\in \xo^{{A_1}^c}$ the
function $$x\to g_2(x,y)g_3(x,y)\dots g_m(x,y)~~(x\in \xo^{A_1})$$ is
$B$-measurable.
This means that
$$\int_{\xo^{A_1}}g_1(x)g_2(x,y)\dots g_m(x,y) d\mu_{A_1}(x)=$$
$$=\int_{\xo^{A_1}}E(g_1|B)(x)g_2(x,y)g_3(x,y)\dots g_m(x,y) d\mu_{A_1}(x)$$
for all $y$ in $\xo_{A_1^c}$.
This completes the proof. \qed \vskip 0.2in

\noindent
Now we finish the proof of the Total Independence Theorem.
We can assume that $|A_i|\geq |A_j|$ whenever $j>i$. Let
  $\chi_i$ be the characteristic function of $S_i$. We have that
$$\mu(S_1\cap S_2\cap\dots\cap S_r)=\int_{\xok}\chi_1\chi_2\dots\chi_{r}
d\mu_{[k]}\,.$$
The Integration Rule shows that
$$\int_{\xok}\chi_i\chi_{i+1}\dots\chi_r\,d\mu_{[k]}=
\int_{\xok}E(\chi_i|\sigma(A_i)^*)\chi_{i+1}\dots\chi_r\,d\mu_{[k]}$$
$$=
\mu(S_i)\int_{\xok}
\chi_{i+1}\chi_{i+2}\dots\chi_r\,d\mu_{[k]}.$$
A simple induction finishes the proof. \qed
\section{The proof of the Euclidean Correspondance Principle}
\label{proofcorr}
\subsection{Random Partitions}\label{randpar}
The goal of this subsection is to prove
the following proposition.
\begin{proposition}
\label{random}
Let $A\subset [k]$ be a subset, then for any $n\geq 1$ there exists
a partition $\xo^A=S_1\cup S_2\cup\dots\cup S_n$, such that
$E(S_i\mid \sigma(A)^*)=\frac{1}{n}$.
\end{proposition}
\begin{proof}
The idea of the proof is that we consider random partitions of $\xo^A$ and show
that by probability one these partitions will satisfy the property of our
proposition.
Let $\Omega=\prod^\infty_{i=1}\{1,2,\dots,n\}^{X_{i,A}}$ be the
set of $\{1,2,\dots,n\}$-valued functions on $\cup^\infty_{i=1} X_{i,A}$.
Each element $f$ of $\Omega$ defines a partition of $X_A$ the following
way. Let
$$ S_f^{i,j}=\{p\in X_{i,A}\,\mid f(p)=j\}\,\,, 1\leq j \leq n,\, i\geq 1\,.$$
$$[\{S^{i,j}_f\}^\infty_{i=1}]= S^j_f\,.$$
Then $\xo^A=S^1_f\cup S^2_f\cup\dots \cup S^n_f$ is our
partition induced by $f$.

\noindent
Note that on $\Omega$ one has the usual Bernoulli probability measure $P$,
$$P(T_{p_1,p_2,\dots,p_r}(i_1,i_2,\dots,i_r))=\frac{1}{n^r}\,,$$
where
$$T_{p_1,p_2,\dots,p_r}(i_1,i_2,\dots,i_r)=\{f\in\Omega\,\mid\, f(p_s)=i_s\,
\,1\leq s \leq r\}\,.$$
A {\bf cylindric intersection set} $T$ in $X_{i,A}$
is a set $T=\cap_{C,C\subsetneq A} T_C$, where  $T_C\subset X_{i,C}$.
First of all note that the number of different cylindric intersection
sets in $X_{i,A}$ is not greater than
$$\prod_{C,C\subsetneq A} 2^{|X_{i,C}|}\leq 2^{(|X_i|^{|A|-1})2^k}\,.$$
Let $0\leq \e \leq\frac{1}{10n}$ be a real number and $T$ be a cylindric
intersection set of elements of size at least $\e |X_{i,A}|\,.$ By the
Chernoff-inequality the probability that an $f\in \Omega$ takes the
value $1$ more than $(\frac{1}{n}+\e)|T|$-times or less than
$(\frac{1}{n}-\e)|T|$-times on the set $T$ is less than
$2\exp(- c_{\e}|T|)$, where the positive constant $c_{\e}$ depends only
on $\e$. Therefore the probability that there exists a cylindric intersection
set $T\subset X_{i,A}$ of size at least $\e |X_{i,A}|$ for which
$f\in \Omega$ takes the
value $1$ more than $(\frac{1}{n}+\e)|T|$-times or less than
$(\frac{1}{n}-\e)|T|$-times on the set $T$ is less than
$$2^{(|X_i|^{|A|-1})2^k} 2 \exp(-c_{\e}\e |X_i|^{|A|})\,.$$
Since $|X_1|<|X_2|<\dots$
 by the Borel-Cantelli lemma we have the following lemma.
\begin{lemma}
For almost all $f\in\Omega$ the following holds: If $\epsilon>0$, then
 there exist only finitely many $i$
such that there exists at least one cylindric intersection set
$T\subset X_{i,A}$ for which $f\in \Omega$ takes the
value $1$ more than $(\frac{1}{n}+\e)|T|$-times or less than
$(\frac{1}{n}-\e)|T|$-times on the set $T$.
\end{lemma}
Now let us consider a cylindric intersection set $Z\subseteq \xo^{A}$,
$Z=\cap_{C,C\subsetneq A} Z_C, \, Z_C\subset \xo^C$.
By the previous lemma, for almost all $f\in \Omega$,
$$\mu(S^1_f\cap Z)=\frac{1}{n}\mu(Z)\,.$$
Therefore for almost all $f\in\Omega$:
$$\mu(S^1_f\cap Z')=\frac{1}{n}(\mu(Z'))\,,$$
where $Z'$ is a finite disjoint union of cylindric intersection sets
in $\xo^{A}$. Consequently, for almost all $f\in\Omega$,
$$\mu(S^1_f\cap Y)=\frac{1}{n}(\mu(Y))\,,$$
where $Y\in \sigma(A)^*$. This shows immediately that
$E(S^1_f\mid \sigma(A)^*)=\frac{1}{n}$ for almost all $f\in\Omega$. Similarly,
$E(S^i_f\mid \sigma(A)^*)=\frac{1}{n}$ for almost all $f\in\Omega$, thus our
proposition follows.
\qed \end{proof} \vskip 0.2in
\noindent
{\bf Remark:} Later on we need a simple modification of our proposition.
Let $\{q_i\}^n_{i=1}$ be non-negative real numbers, such that
$\sum_{i=1}^n q_i=1$. Repeat the construction of the measure on $\Omega$
as in Proposition \ref{random} with the exception that for any $p\in X_{i,A}$
the probability that $f(p)=i$ is $q_i$ instead of $\frac{1}{n}$. Then
with probability one $E(S^i_f\,|\, \sigma (A)^*)=q_i$.
\subsection{Independent Complement in Separable $\sigma$-algebras}

Let $\mathcal{A}$ be a separable $\sigma$-algebra on a set $X$, and
let $\mu$ be a probability measure on $\mathcal{A}$. Two sub
$\sigma$-algebras $\mathcal{B}$ and $\mathcal{C}$ are called
independent if $\mu(B\cap C)=\mu(B)\mu(C)$ for every
$B\in\mathcal{B}$ and $C\in \mathcal{C}$. We say that $\mathcal{C}$
is an {\it independent complement} of $\mathcal{B}$ in $\mathcal{A}$
if it is independent from $\mathcal{B}$ and
$\langle\mathcal{B},\mathcal{C}\rangle$ is dense in $\mathcal{A}$.

\begin{definition} Let $\mathcal{A}\geq\mathcal{B}$ be two
  $\sigma$-algebras on a set $X$ and let $\mu$ be a probability measure on
  $\mathcal{A}$.
 A $\mathcal{B}$-random $k$-partition in $\mathcal{A}$ is
  a partition $A_1,A_2,\dots,A_k$ of $X$ into $\mathcal{A}$-measurable sets
  such that $E(A_i|\mathcal{B})=1/k$ for every $i=1,2,\dots,k$.
\end{definition}

\begin{theorem}[Independent Complement]\label{incom}
 Let $\mathcal{A}\geq\mathcal{B}$ be two separable $\sigma$-algebras on a set
  $X$ and let $\mu$ be a probability measure on $\mathcal{A}$. Assume that for
  every natural number $k$ there exists a $\mathcal{B}$-random $k$-partition
  $\{A_{1,k},A_{2,k},\dots,A_{k,k}\}$ in $\mathcal{A}$. Then there is
  an independent complement $\mathcal{C}$ of $\mathcal{B}$ in $\mathcal{A}$.
 (Note that this is basically the Maharam-lemma, see \cite{Mah})
\end{theorem}

\proof
Let $S_1,S_2,\dots$ be a countable generating system of $\mathcal{A}$ and
  let $\mathcal{P}_k$ denote the finite Boolean
 algebra generated by $S_1,S_2,\dots,S_k$ and
  $\{A_{i,j}|i\leq j\leq k\}$. Let $\mathcal{P}_k^*$ denote the atoms of
  $\mathcal{P}_k$. It is clear that for every atom $R\in\mathcal{P}_k^*$ we
  have that $E(R|\mathcal{B})\leq 1/k$ because $R$ is contained in one of the
  sets $A_{1,k},A_{2,k},\dots,A_{k,k}$. During the proof we fix one
  $\mathcal{B}$-measurable version of $E(R|\mathcal{B})$ for every $R$. The
  algebra $\mathcal{P}_{k}$ is a subalgebra of $\mathcal{P}_{k+1}$ for every
  $k$. Thus we can define total orderings on the sets $\mathcal{P}_k^*$ such
  a way that if $R_1,R_2\in\mathcal{P}_k^*$ with $R_1<R_2$ and
  $R_3,R_4\in\mathcal{P}_{k+1}^*$ with $R_3\subseteq R_1, R_4\subseteq R_2$
  then $R_3<R_4$. We can assume that
  $\sum_{R\in\mathcal{P}_k^*}E(R,\mathcal{B})(x)=1$ for any element $x\in X$.
  It follows that for $k\in\mathbb{N}$, $x\in X$ and
  $\lambda\in [0,1)$ there is a
unique element $R(x,\lambda,k)\in\mathcal{P}_k^*$
  satisfying
$$\sum_{R<R(x,\lambda,k)}E(R|\mathcal{B})(x)\leq \lambda$$
and
$$\sum_{R\leq R(x,\lambda,k)}E(R|\mathcal{B})(x)>\lambda.$$
For an element $R\in\mathcal{P}_k^*$ let $T(R,\lambda,k)$ denote the
set of those points $x\in X$ for which $R(x,\lambda,k)=R$. It is
easy to see that $T(R,\lambda,k)$ is $\mathcal{B}$-measurable. Let
us define the $\mathcal{A}$-measurable set $S(\lambda,k)$ by
$$S(\lambda,k)=\bigcup_{R\in\mathcal{P}_k^*}(T(R,\lambda,k)
\cap(\cup_{R_2<R}R_2))$$
and $S'(\lambda,k)$ by
$$S'(\lambda,k)=\bigcup_{R\in\mathcal{P}_k^*}(T(R,\lambda,k)\cap(\cup_{R_2\leq
  R}R_2)).$$
Note that
$$S(\lambda,k)=\{x\in X\,\mid\,\sum_{R_2\leq R_k(x)} E(R\mid \cB)(x)\leq
\lambda\}\,,$$
where $R_k(x)$ is the element of $\mathcal{P}_k^*$ that contains $x$.
\begin{proposition} \label{propindep}
\begin{description}
\item[(i)] $\lambda-\frac{1}{k}\leq
E(S(\lambda,k)\mid \cB)(x)\leq \lambda$ for any $x\in X$.
\item[(ii)] If $k<t$, then
$S(\lambda,k)\subseteq S(\lambda,t)\subseteq S'(\lambda,k)\,.$
\item[(iii)] $E(S'(\lambda,k)\backslash S(\lambda,k)\mid\cB)(x)\leq
  \frac{1}{k} $ for any $x\in X$.
\end{description}
\end{proposition}
\proof
First observe that
$$ \lambda-\frac{1}{k}\leq \sum_{R< R(x,\lambda,k)}E(R\mid \cB)(x)\leq
\lambda\,,$$
for any $x\in X$. Also, we have
\begin{equation} \label{apr30}
S(\lambda,k)=\bigcup_{R,R_1\in \mathcal{P}_k^*, R<R_1}
(R\cap T(R_1,\lambda,k)),\quad
S'(\lambda,k)=\bigcup_{R,R_1\in \mathcal{P}_k^*, R\leq R_1}
(R\cap T(R_1,\lambda,k)). \end{equation}
That is by the basic property of the conditional expectation:
$$E(S(\lambda,k)\mid \cB)=
\sum_{R,R_1\in \mathcal{P}_k^*, R<R_1} E(R\mid \cB)
\chi_{T(R_1,\lambda,k)}\,.$$
That is
\begin{equation} \label{egy5}
E(S(\lambda,k)\mid \cB)(x)=\sum_{R<R(x,\lambda,k)}  E(R\mid \cB)(x)\,.
\end{equation}
and similarly
\begin{equation} \label{egy5b}
E(S'(\lambda,k)\mid \cB)(x)=\sum_{R\leq R(x,\lambda,k)}  E(R\mid \cB)(x)\,.
\end{equation}
Hence (i) and (iii) follows immediately, using the fact that
$E(R'\mid\cB)\leq \frac{1}{k}$ for any $R'\in \mathcal{P}_k^*$.

\noindent
Observe that for any $R\in \mathcal{P}_k^*$,
$T(R,\lambda,k)=\cup_{R'\subseteq R, R'\in \mathcal{P}_t^*}
T(R',\lambda,t)\,.$
Hence
$$\bigcup_{R,R_1\in \mathcal{P}_k^*, R<R_1}
(R\cap T(R_1,\lambda,k)) \subseteq
\bigcup_{R',R'_1\in \mathcal{P}_t^*, R'<R'_1}
(R'\cap T(R'_1,\lambda,t)) \subseteq$$ $$\subseteq
\bigcup_{R,R_1\in \mathcal{P}_k^*, R\leq R_1}
(R\cap T(R_1,\lambda,k))$$
Thus (\ref{apr30}) implies (ii)\,.
 \qed

\vskip 0.2in
\noindent
\begin{lemma} Let $S(\lambda)=\cup^\infty_{k=1} S(\lambda,k)\,.$
 Then if $\lambda_2<\lambda_1$, then $S(\lambda_2)\subseteq S(\lambda_1)$.
\end{lemma}
\proof Note that $x\in S(\lambda_2,k)$ if and only if
$x\in R_2$ for some $R_2< R(x,\lambda_2,k)\,.$ Obviously,
$R(x,\lambda_2,k)< R(x,\lambda_1,k)$, thus $x\in S(\lambda_1,k)$. Hence
$S(\lambda_2)\subseteq S(\lambda_1)$ \qed
\begin{lemma} $E(S(\lambda)\mid \cB)=\lambda$.
\end{lemma}
\proof
Since $\chi_{S(\lambda,k)}\stackrel
{L_2(X,\mu)}{\to} \chi_{S(\lambda)}$, we have
$E(S(\lambda,k)\mid \cB)
\stackrel{L_2(X,\mu)}{\to}E(S(\lambda)\mid \cB)$. That is
by (i) of Proposition \ref{propindep} $E(S(\lambda)\mid \cB)=\lambda$. \qed
\vskip 0.2in
\noindent
The last two lemmas together imply that the sets $S(\lambda)$ generate a
$\sigma$-algebra $\mathcal{C}$ which is independent from $\mathcal{B}$.

\noindent
Now we have to show that $\mathcal{B}$ and $\mathcal{C}$ generate
$\mathcal{A}$. Let $S\in\mathcal{P}_k$  for some
$k\in\mathbb{N}$. We say that $S$ is an interval if there exists an element
$R\in\mathcal{P}_k^*$ such that $S=\cup_{R_1\leq R}R_1$.  It is enough to show
that any interval $S\in\mathcal{P}_k$ can be generated by $\mathcal{B}$ and
$\mathcal{C}$.

\noindent
Suppose that $\{T_t\}^\infty_{t=1}$ be sets in $\langle\cB,\cC\rangle$
such that $T_t\subset S$ and
$\|E(S\mid \cB)-E(T_t\mid \cB)\|$ tends uniformly to $0$
as $t\to\infty$. Then $\mu(S \backslash T_t)\to 0$
as $t\to \infty$, that is $\cB$ and $\cC$ generate $S$.
Indeed,
$$\mu(S\backslash T_t)=\int_X(\chi_S- \chi_{T_t})=\int_X(E(S\mid \cB)- 
E(T_t\mid \cB))\,.$$
So let $t\geq k$ be an arbitrary natural number. It is clear that
$S$ is an interval in $\mathcal{P}_t$. For a natural number $0\leq d\leq t-1$
let $F_d$ denote the $\mathcal{B}$-measurable set on which $E(S|\mathcal{B})$
is in the interval $(\frac{d}{t},\frac{d+1}{t}]$. Now we approximate $S$ by
$$T_t=\bigcup_{d=0}^{t-1}(F_d\cap S(\frac{d}{t}))\in
\langle\mathcal{B},\mathcal{C}\rangle.$$
\begin{lemma} \label{lemma63}
$T_t\subseteq S$.
\end{lemma}
\proof
It is enough to prove that $F_d\cap S(\frac{d}{t},k)\subset S$ for any
$0\leq d \leq t-1$, $t<k$.
Observe that
$$F_d=\{x\in X\,\mid\, \frac{d}{t}<
\sum_{R_1\leq R} E(R_1\mid \cB)(x)\leq \frac{d+1}{t}\}$$
and
$$S(\frac{d}{t},k)=\{x\in X\,\mid\, \sum_{R_2\leq R_k(x)} E(R_2\mid
\cB)(x)\leq \frac{d}{t}\}\,.$$
Thus if $x\in F_d\cap S(\frac{d}{t},k)$ then $x\in S$.\qed
\begin{lemma} \label{lemma64} For any $x\in X$,
$$\left|E(S\mid \cB)(x)-E(T_t\mid \cB)(x)\right|\leq \frac{3}{t}\,.$$
\end{lemma}
\proof
First note that by Proposition \ref{propindep} (iii)
\begin{equation}\label{egyx}
\left|E(S(\frac{d}{t})\mid \cB)(x)- E(S(\frac{d}{t},t)\mid \cB)(x)\right|\leq
\frac{1}{t}\,.
\end{equation}
Note that
$$E(T_t\mid \cB)(x)=\sum^{t-1}_{d=0} \chi_{F_d}(x)
E(S(\frac{d}{t})\mid\cB)(x)\,.$$
Suppose that $x\in F_d$. Then by (\ref{egy5}) and (\ref{egyx}),
$$\left|E(T_t\mid \cB)(x)-\sum_{R'<R(x,\frac{d}{t},t)}
E(R'\mid \cB)(x)\right|\leq
\frac{1}{t}\,. $$
On the other hand
$E(S\mid\cB)(x)=\sum_{R'\leq R} E(R'\mid \cB)(x)$ and
$\frac{d}{t}\leq \sum_{R'\leq R} E(R'\mid \cB)(x) < \frac{d+1}{t}\,.$
That is
$$\left|E(S\mid \cB)(x)-E(T_t\mid \cB)(x)\right|\leq \frac{3}{t}
\,.\quad \qed $$
The Theorem now follows from Lemma \ref{lemma64} immediately. \qed

\begin{definition} Let $(X,\mathcal{A},\mu)$ be a probability space, and
  assume that a finite group $G$ is acting on $X$ such that $\mathcal{A}$ is
  $G$-invariant as a set system. We say that the action of $G$ is free if
  there is a subset $S$ of $X$ with $\mu(S)=1/|G|$ such that $S^{g_1}\cap
  S^{g_2}=\emptyset$ whenever $g_1$ and $g_2$ are distinct elements of $G$.
\end{definition}

We will need the following consequence of Theorem \ref{incom}.

\begin{lemma}\label{incom2} Let $\mathcal{A}\geq\mathcal{B}$ be two separable
  $\sigma$-algebras on the set $X$ and let $\mu$ be a probability measure on
  $\mathcal{A}$. Assume that a finite group $G$ is acting on $X$ such that
  $\mathcal{A},\mathcal{B}$ and $\mu$ are $G$ invariant. Assume furthermore
  that the action of $G$ on $(X,\mathcal{B},\mu)$ is free and for any $k>1$
  there exists a
  $\mathcal{B}$-random $k$ partition of $X$ in $\mathcal{A}$
. Then there is an independent complement $\mathcal{C}$ in
  $\mathcal{A}$ for $\mathcal{B}$ such that $\mathcal{C}$ is
  elementwise $G$-invariant.
\end{lemma}

\begin{proof} Let $S\in\mathcal{B}$ be a set showing that $G$ acts freely
  on $\mathcal{B}$. Let $\mathcal{A}|_S$ and
  $\mathcal{B}|_S$ denote the restriction of $\mathcal{A}$ and
  $\mathcal{B}$ to the set $S$. It is clear that if
  $\{A_1,A_2,\dots,A_k\}$ is a $\mathcal{B}$-random $k$-partition in
  $\mathcal{A}$ then $\{S\cap A_1,S\cap A_2,\dots,S\cap A_k\}$ is a
  $\mathcal{B}|_S$-random $k$ partition in $\mathcal{A}|_S$. Hence by
  Theorem \ref{incom} there exists an independent complement $\mathcal{C}_1$ of
  $\mathcal{B}|_S$ in $\mathcal{A}|_S$. The set
$$\mathcal{C}=\{\bigcup_{g\in G} H^g|H\in\mathcal{C}_1\}$$ is a
$\sigma$-algebra because the action of $G$ is free. Note that the
elements of $\mathcal{C}$ are $G$-invariant. Since $E(\cup_{g\in G}
H^g|\mathcal{B})=\sum_{g\in G} E(H|\,\mathcal{B}|_S)^g$ we obtain that
the elements of $\mathcal{C}$ are independent from $\mathcal{B}$. It
is clear that $\langle\mathcal{C},\mathcal{B}\rangle$ is dense in
$\mathcal{A}$. \qed
\end{proof}

\subsection{Separable Realization}\label{sepreal}
In this subsection we show how to pass from nonseparable $\sigma$-algebras to
separable ones.

First note that the symmetric group $S_k$ acts on the space $\xo^k$ by
permuting the coordinates:
$$(x_1,x_2,\dots,x_k)^\pi=(x_{\pi^{-1}(1)}, x_{\pi^{-1}(2)},\dots,
x_{\pi^{-1}(k)})\,.$$
The group also acts on the subsets of $[k]$ and $\sigma(A)^\pi=
\sigma(A^\pi)$, where $A^\pi$ denotes the image of the subset $A$
under $\pi\in S_k$. We will denote by $S_A$ the symmetric group
acting on the subset $A$.

\begin{definition}\label{sepre} A {\bf separable system} on
  $\xo^k~~, r\leq k$ is a system of atomless separable $\sigma$-algebras
  $\{l(A)~|~A\in r([k])\}$ and
functions $\{F_A:\xo^k\to
  [0,1]~|~A\in r([k])\}$ with the
following properties
\begin{enumerate}
\item $l(A)$ is a subset of $\sigma(A)$ and is independent
from $\sigma(A)^*$ for every $\emptyset\neq
  A\subseteq [k]$.
\item $l(A)^\pi=l(A^\pi)$ for every permutation $\pi\in S_k$.
\item $S^\pi=S$ for every $S\in l(A)$ and $\pi\in S_A$.
\item $F_A$ is an $l(A)$-measurable function which defines a measurable
equivalence
between the measure algebras of $(\xo^k,l(A),\muo^k)$ and $[0,1]$. (see
Appendix)
\item $F_A({\bf x})=F_{A^\pi}({\bf x}^\pi)$ for every element
${\bf x}\in\xo^k~,~\pi\in S_k$ and
$A\subseteq [k]$.
\end{enumerate}
\end{definition}

 The main proposition in this section is the following one.

\begin{proposition}\label{reali} For every separable $\sigma$-algebra
$\cA$ in $\sigma([k])$ there exists
 a separable system such that for every set $M\in\cA$
there is a set $Q\in\langle l(A)~|~A\in r([k])\rangle$ with
$\mu_{[k]}(M\triangle Q)=0$.
\end{proposition}

This proposition immediately implies Theorem \ref{eucc} since the map
$F:\xo^k\mapsto [0,1]^{r([k])}$ whose coordinate functions are $\{F_A~|~A\in
r([k])\}$ constructed in Proposition \ref{reali}
is a separable realization.

We will need the following three lemmas.

\begin{lemma}\label{sep1} Let $\mathcal{B}\subseteq\mathcal{A}$ two
$\sigma$-algebras on a
  set $Y$, and let $\mu$ be a probability measure on $\mathcal{A}$. Then for
  any separable sub-$\sigma$-algebra $\bar{\mathcal{A}}$ of $\mathcal{A}$
  there exists
 a separable sub $\sigma$-algebra $\bar{\mathcal{B}}$ of $\mathcal{B}$
  such that $E(A|\mathcal{B})=E(A|\bar{\mathcal{B}})$
for every $A\in\bar{\mathcal{A}}$.
\end{lemma}

\proof
  We use the fact that $\bar{\mathcal{A}}$ is a
  separable metric space with the distance $d(A,B)=\mu(A\triangle B)$. Let
  $W=\{D_1,D_2,\dots\}$ be a countable dense subset of $\bar{\mathcal{A}}$
  with the previous distance. Let $C_{p,q}^i=E(D_i\mid \cB)^{-1}(p,q)$,
where $p<q$ are rational numbers. Clearly, $E(D_i\mid \cB)$ is a
$\cB_i$-measurable function, where $\cB_i=\langle C_{p,q}^i\mid
p<q\in \bQ\rangle$. Obviously, $E(D_i\mid \overline{\cB})=
E(D_i\mid \cB)$ for any $i\geq 1$, where
$\overline{\cB}=\langle \cB_i\mid i=1,2,\dots\rangle\,.$
Now observe that $E(D_i\mid \cB)\stackrel {L_2}{\to}
 E(D,\cB)$ and $E(D_i\mid \overline{\cB})\stackrel {L_2}{\to}
 E(D,\overline{\cB})$
if $D_i\to D$. Hence for any $D\in \overline{\cA}$, $E(D\mid\overline{\cB})=
E(D\mid\cB)$.\qed
\begin{lemma}\label{inv1} Let $A\subseteq[k]$ be a subset and assume that
  there are atomless separable $\sigma$-algebras
$d(\{i\})\subset\sigma(\{i\})$\,,$i\in A$
  such that $d(\{i\})^\pi=d(\{i^\pi\})$ for every $i\in A$ and $\pi\in
  S_A$. Then $S_A$ acts freely on $\langle d(\{i\})|i\in A\rangle$.
\end{lemma}

\proof The permutation invariance implies that there exists a
  $\sigma$-algebra $\mathcal{A}$ on $\xo$ such that
  $P_{\{i\}}^{-1}(\mathcal{A})=d(\{i\})$ for every $i\in A$.
  Let $F:\xo\to [0,1]$ be a $\mathcal{A}$-measurable measure
preserving map. Now
  we can define the map $G:\xo^A\to [0,1]^{A}$ by
  $$G(x_{i_1},x_{i_2},\dots,x_{i_{|A|}}):=
(F(x_{i_1}),F(x_{i_2}),\dots,F(x_{i_{|A|}})).$$
Let us introduce $S':=\{(y_1,y_2,\dots,y_r)|y_1<y_2<\dots< y_r\}\subset[0,1]^A$
and $S:=G^{-1}(S')$. Clearly $\muo^A(S)=1/|A|!$
and $S^\pi\cap S^\rho=\emptyset$
for every two different elements $\pi\neq\rho$ in $S_A$.
\qed

\begin{lemma}\label{sep2} Let $k$ be a natural number and assume that for every
  $A\subseteq[k]$ there is a separable $\sigma$-algebra $c(A)$ in
  $\sigma(A)$. Then for every $A\subseteq[k]$ there is a separable
  $\sigma$-algebra $d(A)$ in $\sigma(A)$ with $c(A)\subseteq d(A)$ such that

\begin{enumerate}
\item  $E(R|\langle d(B)|B\in A^*\rangle)=E(R|\sigma(A)^*)$ whenever $R\in
  d(A)$.
\item $d(A)^\pi=d(A^\pi)$ for every element $\pi\in S_k$.
\item $d(B)\subseteq d(A)$ whenever $B\subseteq A$
\end{enumerate}
\end{lemma}

\proof First we construct algebras $d'(A)$
  recursively. Let $d'([k])$ be
  $\langle c([k])^\pi|\pi\in S_k\rangle$.
Assume that we have already constructed the algebras $d'(A)$ for
  $|A|\geq t$. Let $A\subseteq[k]$ be such that $|A|=t$. By Lemma \ref{sep1}
  we can see that there exists a separable subalgebra $\widetilde{d'(A)}$ of
$\sigma(A)^*$ such that
  $E(R|\sigma(A)^*)=E(R|\widetilde{d'(A)})$ for every $R\in d'(A)$.
Since $\sigma(A)^*$ is
  generated by the algebras $\{\sigma(B)|B\in A^*\}$ we have that
  every element of $\sigma(A)^*$ is a countable expression of some
  sets in these algebras. This implies that any separable sub $\sigma$-algebra
  of $\sigma(A)^*$ is generated by separable sub
  $\sigma$-algebras of the algebras $\sigma(B)$ where $B\in A^*$.
In particular we can choose
 separable $\sigma$-algebras $d'(A,B)\supset c(B)$ in
$\sigma(B)$ for every $B\in A^*$ such
 that $\langle d'(A,B)|B\in A^*\rangle\supseteq \widetilde{d'(A)}$. For a set
 $B\subseteq[k]$ with $|B|=t-1$ we define $d'(B)$ as the $\sigma$-algebra
 generated by all the algebras in the form of $d'(C,D)^\pi$,
where $\pi\in S_k$ , $D^\pi=B$ ,
 $|C|=|D|+1$ and $D\subseteq C$. Since
 $d'(C,D)^\pi\subseteq\sigma(D)^\pi=\sigma(B)$ we have that
 $d'(B)\subseteq\sigma(B)$. Furthermore we have that
 $d'(B)^\pi$=$d'(B^\pi)$ for every $\pi\in S_k$.

Now let $d(A):=\langle d'(B)~|~B\subseteq A\rangle$. the second
requirement in the lemma is trivial by definition. We prove the
first one. The elements of $d(A)$ can be approximated by finite
unions of intersections of the form $\bigcap_{B\subseteq A}T_B$
where $T_B\in d'(B)$ and so it is enough to prove the statement if
$R$ is such an intersection. Let $Q=\bigcap_{B\subset A,B\neq
A}T_B$. Now

$$E(R|\langle d(B)|B\in A^*\rangle)=E(R|\langle d'(B)|B\subset
A,B\neq A\rangle)\,.$$ By the basic property of the conditional expectation
(see Appendix) :
$$E(R|\langle d'(B)|B\subset
A,B\neq A\rangle)=E(T_A|\langle d'(B)|B\subset A,B\neq A\rangle)
\chi_Q=E(T_A|\sigma(A)^*)\chi_Q=$$$$=E(R|\sigma(A)^*).$$
\qed

\medskip
\noindent{\bf Proof of Proposition \ref{reali}}~~ We construct the
algebras $l(A)$ in the following steps. For each non-empty subset
$A\subseteq[k]$ we choose an atomless separable $\sigma$-algebra
$c(A)\subseteq\sigma(A)$ containing a $\sigma(A)^*$-random
$r$-partition for every $r$. We also assume
that $\mathcal{A}\subseteq c([k])$. Applying
Lemma \ref{sep2} for the previous
system of separable
  $\sigma$-algebras $c(A)$ we obtain the $\sigma$-algebras $d(A)$.
By Lemma \ref{inv1} and the permutation invariance property of the
previous lemma, $S_{[r]}$ acts freely on
  $d([r])^*=\langle d(B)|B\in
  [r]^*\rangle$. Hence
using Lemma \ref{incom2}, for every $\emptyset\neq A\in[k]$ we can choose an
independent
 complement $l([r])$ for $d([r])^*$ in $d([r])$ such
that $l([r])$ is elementwise invariant under the
  action of $S_{[r]}$.  The algebras $l([r])$ are independent from
$\sigma([r])^*$
since $\mu(R)=E(R|d([r])^*)=E(R|\sigma([r])^*)$ for every $R\in l([r])$.
Now we define $l(A)$, where $|A|=r$ by $l(A)=l([r])^\pi$ for some
$\pi \in S_k$, $\pi([r])=A$. Note that $l(A)$ does not depend on the choice
of $\pi$. By Lemma \ref{measurealgebra} of the Appendix we have
maps $F_{[r]}:\xo^r\to [0,1]$ such that $F^{-1}$ defines a measure
algebra isomorphism between $\cM([0,1],\cB,\lambda)$ and
$\cM(\xo^r,l[r],\mu^r)$.
Let $F_A=\pi^{-1}\circ F_{[r]}$, where $\pi$ maps $[r]$ to $A$. Again,
$F_{[r]}$ does not depend on the particular choice of the
permutation $\pi$. \qed

\section{Appendix on basic measure theory}
In this section we collect some of the basic results of measure theory
we frequently use in our paper.

\vskip 0.1in
\noindent
\underline{Separable measure spaces:}
Let $(X,\cA,\mu)$ be a probability measure space. Then we call $A,A'\in\cA$
equivalent if $\mu(A\triangle A')=0$. The equivalence classes form
a complete metric space, where $d([A],[B])=\mu(A\triangle B)\,.$
This classes form a Boolean-algebra as well, called the
{\bf measure algebra} $\cM(X,\cA,\mu)$. We say that $(X,\cA,\mu)$ is
a {\bf separable} measure space if
$\cM(X,\cA,\mu)$ is a separable metric
space. It is important to note that if $(X,\cA,\mu)$ is separable and atomless,
then its measure algebra is isomorphic to the measure algebra of
the standard Lebesgue space $([0,1],\cB,\lambda)$, where $\cB$ is
the $\sigma$-algebra of Borel sets (see e.g. \cite{Hal}).
We use the following folklore version of this
theorem.
\begin{lemma} \label{measurealgebra}
If $(X,\cA,\mu)$ is a separable and atomless measure space, then there
exists a map $f:X\to [0,1]$ such that
 $f^{-1}(\cB)\subset \cA$,
$\mu(f^{-1}(U))=\lambda(U)$ for any $U\in \cB$ and
for any $L\in\cA$ there exists
$M\in\cB$ such that $L$ is equivalent to $f^{-1}(M)$.

\noindent
In other words, if $F:[0,1]\to X$ is an injective measure preserving
measure algebra homomorphism
such that the image of the the Borel-algebra is just $\mathcal{A}$, then
$F$ can be {\bf represented} by the map $f$. That is for any 
measurable set $U\subset [0,1]$, $F(U)$ is the set representing $f^{-1}(U)$.
\end{lemma}
\proof
Let $I_0$ denote the interval $[0,\frac{1}{2}]$, $I_1=[\frac{1}{2},1]$.
Then let $I_{0,0}=[0,\frac{1}{4}]$, $I_{0,1}=[\frac{1}{4},\frac{1}{2}]$,
$I_{1,0}=[\frac{1}{2},\frac{3}{4}]$, $I_{1,1}=[\frac{3}{4},1]$.
Recursively, we define the dyadic intervals $I_{\alpha_1,\alpha_2,\dots,
\alpha_k}$, where $(\alpha_1,\alpha_2,\dots,
\alpha_k)$ is a $0-1$-string.
Let $T$ be the Boolean-algebra isomorphism between the measure algebra
of $(X,\cA,\mu)$ and the measure algebra of $([0,1],\cB,\lambda)$.
Then we have disjoint sets $U_0, U_1\in\cA$ such that
$T([U_0])=[I_0]$, $T([U_1])=[I_1]$. Clearly $\mu(X\backslash (U_0\cup U_1)=0$.
Similarly, we have disjoint subsets of $U_0$, $U_{0,0}$ and $U_{0,1}$
such that $T([U_{0,0}])=[I_{0,0}]$ and $T([U_{0,1}])=[I_{0,1}]$.
Recursively, we define $U_{\alpha_1,\alpha_2,\dots,
\alpha_k}\in \cA$ such that
$U_{\alpha_1,\alpha_2,\dots,
\alpha_{k-1},0}$ and $U_{\alpha_1,\alpha_2,\dots,
\alpha_{k-1},0}$ are disjoint and $T([U_{\alpha_1,\alpha_2,\dots,
\alpha_k}])= I_{\alpha_1,\alpha_2,\dots,
\alpha_k}$. For any $k>0$, the set of points in $X$ which are not
included in some
$U_{\alpha_1,\alpha_2,\dots,
\alpha_k}$ has measure zero.
Now define
$$f(p):=\cap^\infty_{k=1} I_{\alpha_1,\alpha_2,\dots,
\alpha_k}\,,$$
where for each $k\geq 1$, $p\in U_{\alpha_1,\alpha_2,\dots,
\alpha_k}$. It is easy to see that $f$ satisfies the conditions
of our lemma. \qed

\vskip 0.1in
\noindent
\underline{Generated $\sigma$-algebras:}
Let $(X,\cC,\mu)$ be a probability measure space and $\cA_1,\cA_2,\dots,\cA_k$
be sub-$\sigma$-algebras.
Then we denote by $\langle \cA_i\mid 1\leq i\leq k\rangle$
the generated $\sigma$-algebra that
is the smallest sub-$\sigma$-algebra of $\cC$ containing the $\cA_i$'s.
Then the equivalence classes  $$[\cup^n_{j=1} (A^j_1\cap A^j_2\cap\dots
\cap A^j_k)]\,,$$ where $A^j_i\in\cA_i$ and $(A^s_1\cap A^s_2\cap\dots
\cap A^s_k)\cap (A^t_1\cap A^t_2\cap\dots
\cap A^t_k)=\emptyset $ if $s\neq t$
 form a dense subset in the measure algebra
$\cM(X,\langle \cA_i\mid 1\leq i\leq k \rangle,\mu)$
with respect to the metric defined above
(see \cite{Hal}).
\vskip 0.1in
\noindent
\underline{Independent subalgebras and product measures:}
The sub-$\sigma$-algebras $\cA_1,\cA_2,\dots,\cA_k\subset \cC$ are
{\bf independent} subalgebras if
$$\mu(A_1)\mu(A_2)\dots\mu(A_k)=\mu(A_1\cap A_2\cap\dots\cap A_k)\,,$$
if $A_i\in \cA_i$.
\begin{lemma}
\label{fremlin}
Let $\cA_1,\cA_2,\dots,\cA_k\subset \cC$ be independent subalgebras
as above and $f_i:X\to [0,1]$ be maps such that $f_i^{-1}$ defines
isomorphisms between the measure algebras $\cM(X,\cA_i,\mu)$ and
$\cM([0,1],\cB,\lambda)$. Then the map $F^{-1}$, $F=\oplus_{i=1}^k f_i:X\to
[0,1]^k$ defines an isomorphism between the measure algebras
$\cM(X,\langle \cA_i\mid 1\leq i\leq k \rangle,\mu)$ and
$\cM([0,1]^k,\cB^k,\lambda^k)$.
\end{lemma}
\proof
Observed that
$$\mu (F^{-1}(\cup_{i=1}^s [A^i_1\times\dots\times A^i_k]))=\sum^s_{i=1}
\lambda^k[A^i_1\times\dots\times A^i_k]$$
whenever $\{A^i_1\times\dots\times A^i_k\}^s_{i=1}$ are disjoint product
sets. Hence $F^{-1}$ defines an isometry between dense subsets of the two
measure algebras. \qed

\vskip 0.1in
\noindent
\underline{Radon-Nykodym Theorem:}
Let $(X,\cA,\mu)$ be a probability measure space and $\nu$ be
an absolutely continuous measure with respect to $\mu$. That is
if $\mu(A)=0$ then  $\nu(A)=0$ as well.
Then there exists an integrable $\cA$-measurable function $f$
such that
$$\mu(A)=\int_A f d\mu$$
for any $A\in\cA$.

\vskip 0.1in
\noindent
\underline{Conditional expectation:}
Let $(X,\cA,\mu)$ be a probability measure space
and $\cB\subset\cA$ be a sub-$\sigma$-algebra. Then by the
Radon-Nykodym-theorem for any integrable
$\cA$-measurable function $f$ there exists an integrable
$\cB$-measurable function $E(f\mid\cB)$ such that
$$\int_B E(f\mid\cB) d\mu=\int_B f d\mu\,,$$
if $B\in\cB$. The function $E(f\mid \cB)$ is called the conditional
expectation of $f$ with respect to $\cB$. It is unique up to a
zero-measure perturbation.
Note that if $a\leq f(x)\leq b$ for almost all $x\in X$, then
$a\leq E(f\mid \cB)(x)\leq b$ for almost all $x\in X$ as well.
Also, if $g$ is a bounded $\cB$-measurable function, then
$$E(fg\mid\cB)=E(f\mid \cB) g\,\,\quad\mbox{almost everywhere}\,.$$
The map $f\to E(f,\cB)$ extends to a Hilbert-space projection
$E:L^2(X,\cA,\mu)\to L^2(X,\cB,\mu)$.

\vskip 0.1in
\noindent
\underline{Lebesgue density theorem:}
Let $A\in \bR^n$ be a measurable set. Then almost all points $x\in A$ is
a {\bf density point}. The point $x$ is a density point if
$$\lim_{r\to 0} \frac{Vol (B_r(x)\cap A)} {Vol (B_r(x))}=1\,,$$
where $Vol$ denotes the $n$-dimensional Lebesgue-measure.

\vskip 0.1in
\noindent
\underline{Coupling:}
Let $A$, $B$ are sets. Let $X$ be an $A$-valued
random variable and $Y$ be a $B$-valued random variable. A {\bf coupling}
of $X$ and $Y$ is a $A\times B$-valued random variable $Z$, such that
the first component of $Z$ has the distribution of $X$ and the second
component of $Z$ has the distribution of $Y$.

\noindent
G\'abor Elek
\noindent
Alfred Renyi Institute of the Hungarian Academy of Sciences
\noindent
POB 127, H-1364, Budapest, Hungary,  elek@renyi.hu

\vskip 0.2in

\noindent
Bal\'azs Szegedy
\noindent
University of Toronto, Department of Mathematics,
\noindent
St George St. 40, Toronto, ON, M5R 2E4, Canada


\begin{thebibliography}{99}
\bibitem{Austin} T. Austin,
{\em On exchangeable random variables and the statistics of large graphs and 
hypergraphs.}
Probability Surveys, 5, (2008), 80-145 (electronic)
\bibitem{Austintao} T. Austin and T. Tao,
{\em On the testability and repair of hereditary hypergraph properties.}
(preprint http://arxiv.org/abs/0801.2179)
\bibitem{Borgs} C. Borgs, J. Chayes, L. Lovasz, V. T. S\'os, B. Szegedy and
K. Vesztergombi,
{\em Graph limits and parameter testing.}
STOC'06: Proceedings of the 38th Annual ACM Symposium on Theory of Computing,
261--270, ACM, New York, 2006.
\bibitem{Chang} C. C. Chang and H. J. Keisler,
{\em Model theory.} Studies in Logic and the Foundations of Mathematics, {\bf73} 
North-Holland Publishing Co., Amsterdam, 1990. 
\bibitem{Gow} T. Gowers,
{\em Quasirandomness, counting and regularity for 3-uniform hypergraphs.}
 Combin. Probab. Comput. {\bf 15} (2006),  no. 1-2, 143--184.
\bibitem{Hal} P. R.  Halmos,
{\em Measure Theory} Van Nostrand Company, Inc., New York, N. Y., 1950.
\bibitem{Ish} Y. Ishigami,
{\em A Simple Regularization of Hypergraphs}

\noindent
preprint
http://arxiv.org/abs/math/0612838
\bibitem{Loeb} P. E. Loeb,
{\em Conversion from nonstandard to standard measure spaces and applications
in probability theory.}
Trans. Amer. Math. Soc.  {\bf 211}  (1975), 113--122.
\bibitem{LSZ} L. Lovasz, B. Szegedy,
{\em Limits of dense graph sequences.}
J. Combin. Theory Ser. B {\bf 96} (2006), no. 6, 933-957.
\bibitem{Lovaszunique} C. Borgs, J. Chayes and L. Lov\'asz,
{\em Moments of Two-Variable Functions and the Uniqueness of Graph Limits}
(preprint)
\bibitem{Mah} D. Maharam,
{\em On homogeneous measure algebras. }
Proc. Nat. Acad. Sci. U. S. A.  {\bf28},  (1942). 108--111.
\bibitem{NRS} B. Nagle, V. R\"odl and M. Schacht,
{\em The counting lemma for regular $k$-uniform hypergraphs.}
Random Structures Algorithms {\bf 28}  (2006),  no. 2, 113--179.
\bibitem{RS} V. R\"odl, M. Schacht,
{\em  Regular partitions of hypergraphs: regularity lemmas.}
 Combin. Probab. Comput. {\bf 16}  (2007),  no. 6, 833--885.
\bibitem{RSko} V. R\"odl, J. Skokan,
{\em Regularity lemma for $k$-uniform hypergraphs.}
Random Structures Algorithms {\bf 25} (2004), no. 1, 1--42.
\bibitem{RStest} V. R\"odl, M. Schacht,
{\em Generalizations of the removal lemma} (preprint)
\bibitem{S}
J. Solymosi,
{\em A note on a question of Erd\"os and Graham. }
Combin. Probab. Comput. {\bf 13}  (2004),  no. 2, 263--267.
\bibitem{Tao} T. Tao,
{\em A variant of the hypergraph removal lemma.}
J. Combin. Theory Ser. A {\bf 113} (2006), no. 7, 1257--1280.




\end{thebibliography}
\end{document}